\documentclass[11pt]{article}

\usepackage{fullpage}
\usepackage{amsthm,amsmath,amsfonts,amssymb,yfonts}
\usepackage[english]{babel}
\usepackage{graphicx}
\usepackage{times}
\usepackage{bm}
\usepackage[numbers,sort&compress]{natbib} 
\usepackage{color}
\usepackage{float}
\usepackage{booktabs}
\usepackage[table,xcdraw]{xcolor}
\usepackage{tcolorbox}
\usepackage{enumitem}
\usepackage[font=small]{caption}
\usepackage{tikz,lipsum,lmodern}
\usepackage{multicol}
\usepackage{tabularx}
\usepackage[utf8]{inputenc}
\usepackage{array}
\usepackage{wrapfig}
\usepackage{multirow}
\usepackage[T1]{fontenc}
\usepackage{dsfont}
\usepackage{authblk}
\usepackage[colorlinks=true,linkcolor=blue,urlcolor=red,citecolor=red]{hyperref}

\definecolor{tropicalrainforest}{rgb}{0.0, 0.46, 0.37}

\theoremstyle{definition}
\newtheorem{definition}{Definition}[section]
\newtheorem{remark}[definition]{Remark}
\newtheorem{example}[definition]{Example}
\newtheorem{problem}[definition]{Problem}
\newtheorem{note}[definition]{Note}

\theoremstyle{plain}
\newtheorem{theorem}[definition]{Theorem}
\newtheorem{lemma}[definition]{Lemma}
\newtheorem{corollary}[definition]{Corollary}
\newtheorem{proposition}[definition]{Proposition}

\DeclareMathOperator{\Mat}{Mat}

\DeclareMathOperator{\mult}{mult}
\DeclareMathOperator{\End}{End}

\title{\textbf{\Large Every $Q$-polynomial distance-regular graph is sharp over $\mathbb{R}$}

	\author{
		Blas Fern\'andez\footnote{IMFM, Jadranska 19, 1000 Ljubljana, Slovenia;
	University of Primorska, UP FAMNIT, Glagoljaška 8, 6000 Koper, Slovenia; Andrej Marušič Institute, University of Primorska, Muzejski trg 2, 6000 Koper, Slovenia. \\ 
	Email: \texttt{blas.fernandez@famnit.upr.si}} \quad
		Jae-Ho Lee\footnote{Department of Mathematics and Statistics, University of North Florida, Jacksonville, FL 32224, USA. \\
		Email: \texttt{jaeho.lee@unf.edu}} \quad
		Jongyook Park\footnote{Department of Mathematics, Kyungpook National University, Daegu, 41566, Republic of Korea. \\
		Email: \texttt{jongyook@knu.ac.kr} (Corresponding author)}
	}}

\begin{document}

\date{}
\maketitle

\begin{abstract}
Let $\Gamma$ denote a distance-regular graph with vertex set $X$ and diameter $D \geq 3$.
Fix a vertex $x \in X$.
Let the field $\mathbb{F}$ be either $\mathbb{R}$ or $\mathbb{C}$.
Let $\Mat_X(\mathbb{F})$ denote the $\mathbb{F}$-algebra of matrices whose rows and columns are indexed by $X$ and all entries in $\mathbb{F}$.
The Terwilliger algebra $T^\mathbb{F} = T^\mathbb{F}(x)$ is the subalgebra of $\Mat_X(\mathbb{F})$ generated by the adjacency matrix $A$ of $\Gamma$ and the dual primitive idempotents $\{E_i^*\}_{i=0}^D$ of $\Gamma$ with respect to $x$.
Let $\{E_i\}_{i=0}^D$ denote the primitive idempotents of $A$.
Assume that the ordering $\{E_i\}_{i=0}^D$ is $Q$-polynomial.
Let $W$ denote an irreducible $T^\mathbb{F}$-module.
We say that $W$ is sharp over $\mathbb{F}$ whenever $\dim (E_r^* W) = 1$, where $r$ is the endpoint of $W$.
It is known, by Nomura and Terwilliger (2008), that every irreducible $T^\mathbb{C}$-module is sharp.
In this paper, we prove that every irreducible $T^\mathbb{R}$-module is sharp.
Once this is established, we obtain four additional results:
(i) if $W$ is an irreducible $T^\mathbb{R}$-module, then its complexification $W^\mathbb{C}= W \otimes_{\mathbb{R}} \mathbb{C}$ is an irreducible $T^\mathbb{C}$-module;
(ii) two irreducible $T^\mathbb{R}$-modules $W_1$ and $W_2$ are isomorphic if and only if their complexifications $W_1^\mathbb{C}$ and $W_2^\mathbb{C}$ are isomorphic as $T^\mathbb{C}$-modules;
(iii) if $\bigoplus_{i=1}^h \Mat_{n_i}(\mathbb{C})$ is the Wedderburn decomposition of $T^\mathbb{C}$, then $\bigoplus_{i=1}^h \Mat_{n_i}(\mathbb{R})$ is the Wedderburn decomposition of $T^\mathbb{R}$;
(iv) each of the subalgebras $E_1^* T E_1^*$, $E_1 T E_1$, $E_D^* T E_D^*$, and $E_D T E_D$ is commutative and every element of these algebras is a symmetric matrix.



\end{abstract}



\section{Introduction}
This paper is concerned with distance-regular graphs \cite{BI1984, BCN}.
These graphs possess a high degree of combinatorial regularity and a rich algebraic structure \cite{Biggs, BCN}.
Distance-regular graphs were introduced by Biggs \cite{Biggs}; notable examples include the Hamming graphs, Johnson graphs, and Grassmann graphs \cite[Chapter~9]{BCN}.
Since their introduction, distance-regular graphs have been extensively studied, and connections with other areas have been found, such as partially ordered sets, representation theory, quantum groups, and orthogonal polynomials; see \cite{BI1984, BCN, DKT2016, Ter2024} and references therein.
The $Q$-polynomial property of a distance-regular graph was introduced by Delsarte in the context of coding theory \cite{Delsarte1973}.
The subconstituent algebra (or Terwilliger algebra) was introduced by Terwilliger in his study of $Q$-polynomial distance-regular graphs \cite{Ter1992JAC, Ter1993JAC-1, Ter1993JAC-2}.
Since then, the importance of the Terwilliger algebra has grown significantly; see \cite{Caughman1999, Egge2000, Suzuki2005JoAC, IT2009MMJ, Lee2013, GK2015, Lee2017, TA2019, TW2020, TKCP2022, BCV2022, Fernandez2022, Huang2025, Morales2025}.

\smallskip
Throughout this paper, let $\mathbb{F}$ denote a field. 
Let $\Gamma$ denote a distance-regular graph with vertex set $X$.
Let $\mathbb{F}^X$ denote the $\mathbb{F}$-vector space consisting of the column vectors whose coordinates are indexed by $X$ and whose entries are all in $\mathbb{F}$. 
Let $\Mat_X(\mathbb{F})$ denote the $\mathbb{F}$-algebra of matrices whose rows and columns are indexed by $X$ and whose entries are all in $\mathbb{F}$.
The algebra $\Mat_X(\mathbb{F})$ acts on $\mathbb{F}^X$ by left multiplication. 
We call $\mathbb{F}^X$ the \emph{standard module}.
In the literature on distance-regular graphs, one typically assumes that the field $\mathbb{F}$ is either $\mathbb{R}$ or $\mathbb{C}$.
Each choice of field has its own advantages and disadvantages depending on the perspective from which $\Gamma$ is studied.
These advantages and disadvantages are discussed in the following paragraphs.

\smallskip
When studying $\Gamma$ from a geometric perspective, it is convenient to assume $\mathbb{F}=\mathbb{R}$ since the standard module $\mathbb{R}^X$ has a Euclidean space structure.
Indeed, working over $\mathbb{R}$ facilitates geometric analysis of the eigenspaces, the cosine sequences corresponding to its eigenvalues, and representations of $\Gamma$ (cf.~\cite[Section~4]{Ter2024}).
Such a geometric analysis is still available over $\mathbb{C}$, but it is generally more cumbersome since the Hermitean dot product on $\mathbb{C}^X$ is more complicated than the dot product on $\mathbb{R}^X$.

\smallskip
When studying $\Gamma$ from the perspective of the Terwilliger algebra, it is natural to work over $\mathbb{C}$.
To describe this perspective, we recall some notation and preliminaries.
Fix a vertex $x \in X$.
The \emph{Terwilliger algebra} $T=T(x)$ of $\Gamma$ is the subalgebra of $\Mat_X(\mathbb{F})$ generated by the adjacency matrix of $\Gamma$ and the dual primitive idempotents of $\Gamma$ with respect to $x$; see Section~\ref{sec:T-alg}.
When it is important to emphasize the field $\mathbb{F}$, we will write $T=T^\mathbb{F}$.
The algebra $T$ is closed under the conjugate-transpose map and hence semisimple \cite[Lemma~3.4]{Ter1992JAC}.
By the Wedderburn theory \cite{CR1962}, the algebra $T$ decomposes as a direct sum of full matrix algebras.
The precise form of this decomposition depends on whether $\mathbb{F}=\mathbb{C}$ or $\mathbb{F}=\mathbb{R}$.
The Wedderburn decomposition of $T^\mathbb{C}$ is described by the $\mathbb{C}$-algebra isomorphism
\begin{equation}\label{eq:Wed decomp over C}
	T^\mathbb{C}  \cong \bigoplus_{i=1}^h \Mat_{n_i}(\mathbb{C}),
\end{equation}
where $h, n_1, \ldots, n_h$ are positive integers and $\Mat_{n_i}(\mathbb{C})$ denotes the $\mathbb{C}$-algebra of all $n_i \times n_i$ matrices with entries in $\mathbb{C}$.
From this decomposition, it follows that every $T^\mathbb{C}$-module decomposes as a direct sum of irreducible $T^\mathbb{C}$-submodules, each corresponding to one of the simple components $\Mat_{n_i}(\mathbb{C})$ in \eqref{eq:Wed decomp over C}.
On the other hand, the Wedderburn decomposition of $T^\mathbb{R}$ is described by the $\mathbb{R}$-algebra isomorphism
\begin{equation}\label{eq:Wed decomp over R}
	T^\mathbb{R}  \cong \bigoplus_{i=1}^{\ell} \Mat_{m_i}(\mathbb{D}_i),
\end{equation}
where $\ell, m_1, \ldots, m_{\ell}$ are positive integers and each $\mathbb{D}_i$ is either $\mathbb{R}$, $\mathbb{C}$, or $\mathbb{H}$, with $\mathbb{H}$ denoting the quaternions.
Comparing the decompositions \eqref{eq:Wed decomp over C} and \eqref{eq:Wed decomp over R}, we can see that \eqref{eq:Wed decomp over R} is relatively more complicated. For this reason, it is natural to work over $\mathbb{C}$.

\smallskip
In the present paper, we assume that $\Gamma$ is $Q$-polynomial.
In this case, every $\mathbb{D}_i$ in \eqref{eq:Wed decomp over R} turns out to be $\mathbb{R}$; this will be shown in Theorem~\ref{thm:Wedderburn}.
This result follows from a property of $\Gamma$ called sharpness.
In this paper, we study the sharpness property of $\Gamma$ and its effect on the structure of $T^{\mathbb{R}}$.
We now summarize our main results after several preliminary remarks concerning $\Gamma$ and its Terwilliger algebra; formal definitions appear in Sections~\ref{sec:DRGs}--\ref{sec:Q-poly Property}.

\smallskip
Throughout the paper, we assume that $\Gamma$ has diameter $D \geq 3$.
Assume that $\mathbb{F}=\mathbb{R}$ or $\mathbb{F}=\mathbb{C}$.
Let $A \in \Mat_X(\mathbb{F})$ denote the adjacency matrix of $\Gamma$.
Let $\{E_i\}_{i=0}^D$ denote the primitive idempotents of $A$.
Assume that the ordering $\{E_i\}_{i=0}^D$ is $Q$-polynomial.
Fix a vertex $x\in X$.
Let $A^*=A^*(x) \in \Mat_X(\mathbb{F})$ denote the diagonal matrix with the $(y,y)$-entry $(A^*)_{y,y} = |X| (E_1)_{x,y}$ for all $y \in X$.
We call $A^*$ the dual adjacency matrix of $\Gamma$ with respect to $x$.
For $0 \leq i \leq D$, let $E_i^*=E_i^*(x) \in \Mat_X(\mathbb{F})$ denote the diagonal matrix with $(y,y)$-entry
\begin{equation*}
	(E_i^*)_{y,y} =
	\begin{cases}
	1 & \text{ if } \quad \partial(x,y) = i,  \\
	0 & \text{ if } \quad \partial(x,y) \neq i 
	\end{cases}
	\qquad (y \in X).
\end{equation*}
The matrices $\{E_i^*\}_{i=0}^D$ are the primitive idempotents of $A^*$.
The Terwilliger algebra $T=T(x)$ of $\Gamma$ is generated by $A$ and $A^*$.
We now recall the notion of a $T$-module.
By a $T$-module, we mean a $T$-submodule of $V = \mathbb{F}^X$.
Define a form $\langle \ , \ \rangle: V \times V \to \mathbb{F}$ such that for $u,v \in V$, 
\begin{equation*}
	\langle u, v \rangle =
	\begin{cases}
	\overline{u}^\top v 	& \text{ if } \quad  \mathbb{F}=\mathbb{C},  \\
	u^\top v 			& \text{ if } \quad \mathbb{F}=\mathbb{R},
	\end{cases}
\end{equation*}
where $\top$ denotes transpose and $-$ denotes complex conjugation.  
By \cite[Lemma 3.4]{Ter1992JAC} $V$ decomposes into an orthogonal direct sum of irreducible $T$-modules.
Moreover, every irreducible $T$-module $W$ is the orthogonal direct sum of the nonvanishing $E_iW$ $(0 \leq i \leq D)$ and the orthogonal direct sum of the nonvanishing $E_i^*W$ $(0 \leq i \leq D)$.

\smallskip
Let $W$ denote an irreducible $T$-module with endpoint $r$, dual endpoint $s$, and diameter $d$; see Section~\ref{sec:T-alg}.
By \cite[Section~3]{Ter1992JAC}, for $0 \leq i \leq D$ we have $E_i^*W \neq 0$ if and only if $r \leq i \leq r+d$.
Moreover, the sum $W = \sum_{i=0}^d E_{r+i}^*W$ is direct and its summands are the eigenspaces of $A^*$ on $W$.
For $0 \leq i \leq D$, we have $E_iW \neq 0$ if and only if $s \leq i \leq s+d$. 
The sum $W = \sum_{i=0}^d E_{s+i}W$ is direct and its summands are the eigenspaces of $A$ on $W$.
For $0 \leq i \leq d$,
\begin{align*}
	AE_{r+i}^*W 	& \subseteq E_{r+i-1}^*W + E_{r+i}^*W + E_{r+i+1}^*W,\\
	A^*E_{s+i}W 	& \subseteq E_{s+i-1}W + E_{s+i}W + E_{s+i+1}W.
\end{align*}
By the above comments, $A$, $A^*$ act on $W$ as a tridiagonal pair; see Definition \ref{def1}.
Define a sequence of scalars $\{\rho_i\}_{i=0}^d$ by $\rho_i = \dim(E_{r+i}^*W)$ for $0 \leq i \leq d$.
By \cite[Corollary~5.7]{ITT}, $\rho_i = \dim(E_{s+i}W)$ for $0 \leq i \leq d$.
By \cite[Corollaries~5.7,~6.6]{ITT}, the sequence $\{\rho_i\}_{i=0}^d$ satisfies $\rho_i = \rho_{d-i}$ for $0 \leq i \leq d$ and $\rho_{i-1} \leq \rho_i$ for $1\leq i \leq \lfloor d/2 \rfloor$.
The sequence $\{\rho_i\}_{i=0}^d$ is called the shape of $W$.
The $T$-module $W$ is called \emph{sharp} whenever $\rho_0=1$ \cite[Definition~1.5]{NT2008LAA2}.
We say $\Gamma$ is \emph{sharp} over $\mathbb{F}$ whenever, for every vertex $y$, each irreducible $T(y)$-module is sharp.

\smallskip
In \cite{NT2008LAA}, Nomura and Terwilliger showed that every irreducible $T^\mathbb{C}$-module is sharp.
Hence $\Gamma$ is sharp over $\mathbb{C}$.
Motivated by this result, it is natural to ask whether $\Gamma$ is also sharp over $\mathbb{R}$.
Although the sharpness of $\Gamma$ over $\mathbb{R}$ has been believed, to the best of our knowledge this fact has not been formally proved or explicitly stated in the literature. 
We now present the first main result of this paper.

\begin{theorem}\label{mainthm}
Every irreducible $T^\mathbb{R}$-module is sharp.
\end{theorem}

\begin{corollary}\label{main cor}
Every $Q$-polynomial distance-regular graph is sharp over $\mathbb{R}$.
\end{corollary}
\noindent
The proofs of the theorem and the corollary appear in Section~\ref{sec:pf thm}.

\smallskip
Building on Theorem~\ref{mainthm}, we now present four additional results.
Let $W$ denote an irreducible $T^\mathbb{R}$-module.
By extending the ground field from $\mathbb{R}$ to $\mathbb{C}$, we turn $W$ into a  $T^\mathbb{C}$-module.
We denote this module by $W^\mathbb{C}$ and call it the \emph{complexification} of $W$; see Section \ref{sec:pf thm2} for the formal definition.
This raises a natural question: is the $T^\mathbb{C}$-module $W^\mathbb{C}$ irreducible?
We answer this question in the following theorem, which is our second main result.

\begin{theorem}\label{thm:complexification}
Let $W$ denote an irreducible $T^\mathbb{R}$-module.
Then the complexification $W^\mathbb{C}$ is an irreducible $T^\mathbb{C}$-module.
\end{theorem}

Our third main result concerns the effect of complexification on the isomorphism classes.

\begin{theorem}\label{thm:WC iso}
Let $W_1$, $W_2$ denote irreducible $T^\mathbb{R}$-modules.
Then $W_1$ and $W_2$ are isomorphic as $T^\mathbb{R}$-modules if and only if $W_1^\mathbb{C}$ and $W_2^\mathbb{C}$ are isomorphic as $T^\mathbb{C}$-modules.
\end{theorem}

Our fourth main result concerns the Wedderburn decompositions of $T^{\mathbb{C}}$ and $T^{\mathbb{R}}$.

\begin{theorem}\label{thm:Wedderburn}
We refer to the Wedderburn decompositions of $T^{\mathbb{C}}$ in \eqref{eq:Wed decomp over C} and $T^\mathbb{R}$ in \eqref{eq:Wed decomp over R}.
Then 
$$
	\ell = h, 
	\qquad m_i = n_i, 
	\qquad
	\mathbb{D}_i = \mathbb{R} \qquad (1 \le i \le \ell).
$$
\end{theorem}
\noindent
The proofs of Theorems~\ref{thm:complexification}, \ref{thm:WC iso}, and \ref{thm:Wedderburn} appear in Section~\ref{sec:pf thm2}.

\smallskip
Our last main result is motivated by \cite{Ter1993SuzukiNote} and stated below.
\begin{theorem}\label{thm:four algebras}
For either $\mathbb{F}=\mathbb{R}$ or $\mathbb{F}=\mathbb{C}$ the following algebras
\begin{equation}\label{eq: four algebras}
	E_1^*TE_1^*, \qquad
	E_1TE_1, \qquad
	E_D^* T E_D^*, \qquad
	E_D T E_D
\end{equation}
are commutative.
Moreover, every element of these algebras is a symmetric matrix.
\end{theorem}
\noindent
The proof of Theorem \ref{thm:four algebras} appears in Section \ref{sec:E*DTE*D}.

\smallskip
This paper is organized as follows.
In Section~\ref{sec:prelim}, we provide some preliminaries.
In Section~\ref{sec:TDpairs}, we review tridiagonal pairs and tridiagonal systems.
In Sections~\ref{sec:DRGs}--\ref{sec:Q-poly Property}, we discuss distance-regular graphs and related topics, including the Bose-Mesner algebra, the dual Bose-Mesner algebra, the Terwilliger algebra, and the $Q$-polynomial property. 
In Section~\ref{sec:T-mod TDpairs}, we describe the action of the matrices $A, A^*$ on an irreducible $T$-module.
In Section~\ref{sec:pf thm}, we prove Theorem~\ref{mainthm}.
In Section~\ref{sec:pf thm2}, we prove Theorems~\ref{thm:complexification}, \ref{thm:WC iso}, and~\ref{thm:Wedderburn}.
In Section~\ref{sec:E*1TE*1}, we discuss the algebras $E_1^*TE_1^*$ and $E_1TE_1$.
In Section~\ref{sec:E*DTE*D}, we examine the algebras $E_D^*TE_D^*$ and $E_DTE_D$ and prove Theorem~\ref{thm:four algebras}.

\section{Preliminaries}\label{sec:prelim}

In this section, we review some notation and basic concepts.
In this section and the next, we assume that $\mathbb{F}$ is an arbitrary field.
All algebras and vector spaces considered in this paper are defined over $\mathbb{F}$.
Let $V$ denote a vector space of finite positive dimension. 
Let $\operatorname{End}(V)$ denote the algebra consisting of the $\mathbb{F}$-linear maps from $V$ to $V$. 
For $A \in \operatorname{End}(V)$ and a subspace $W \subseteq V$, we say that $W$ is an \emph{eigenspace} of $A$ whenever $W$ is nonzero and there exists $\theta \in \mathbb{F}$ such that $W = \{v \in V \mid Av = \theta v\}$.
In this case, $\theta$ is called the \emph{eigenvalue} of $A$ associated with $W$. 
The map $A$ is said to be \emph{diagonalizable} whenever $V$ is spanned by the eigenspaces of $A$.
Assume that $A$ is diagonalizable.
Let $\{V_i\}^d_{i=0}$ denote an ordering of the eigenspaces of $A$ and let $\{\theta_i\}_{i=0}^d$ denote the corresponding ordering of the eigenvalues of $A$.
By linear algebra,
\begin{equation}\label{V:ds}
	V = \sum_{i=0}^d V_i \qquad \qquad (\text{direct sum}).
\end{equation}
For $0 \leq i \leq d$, define $E_i \in \operatorname{End}(V)$ such that 
\begin{equation}\label{def:E_i eq}
	(E_i - I )V_i = 0, \qquad 
	E_iV_j  = 0 \quad  (0 \leq j \leq d, \ j \neq i), 
\end{equation}
where $I$ is the identity element in $\operatorname{End}(V)$.
We refer to $E_i$ as the \emph{primitive idempotent} of $A$ associated with $V_i$.
Observe that
\begin{equation}\label{Vi=EiV}
	V_i = E_iV \qquad (0 \leq i \leq d).
\end{equation}
Moreover,
\begin{align}
	&& & I = \sum_{i=0}^d E_i, \qquad \qquad E_iE_j = \delta_{i,j}E_i && (0 \leq i, j \leq d), && \label{eq(1):I,E_i}\\
	&& & A = \sum^d_{i=0}\theta_iE_i, \qquad \quad E_i = \prod_{\substack{0 \leq j \leq d \\ j \neq i}} \frac{A - \theta_j I}{\theta_i - \theta_j} && (0 \leq i \leq d), && \label{eq(2):I,E_i}\\
	&& & AE_i = E_iA = \theta_iE_i && (0 \leq i \leq d). && \label{eq(3):I,E_i}
\end{align} 
We note that each of $\{E_i\}_{i=0}^d$ and $\{A^i\}_{i=0}^d$ forms a basis for the subalgebra of $\operatorname{End}(V)$ generated by $A$.

\section{Tridiagonal pairs and tridiagonal systems}\label{sec:TDpairs}

In this section, we review the concepts of tridiagonal pairs and tridiagonal systems.

\begin{definition}[{\cite[Definition~1.1]{ITT}}]\label{def1}
Let $V$ denote a vector space of finite positive dimension. 
By a \emph{tridiagonal pair} (or \emph{TD pair}) on $V$, we mean an ordered pair $A, A^*$ of elements in $\operatorname{End}(V)$ that satisfy (i)--(iv) below.
\begin{enumerate}[label=(\roman*),font=\rm,itemsep=0mm]
	\item Each of $A$ and $A^*$ is diagonalizable on $V$. 
	\item There exists an ordering $\{ V_i \}_{i=0}^d$ of the eigenspaces of $A$ such that 
	$$
	A^*V_i \subseteq V_{i-1}+V_i+V_{i+1} \qquad (0 \leq i \leq d),
	$$ 
	where $V_{-1}=V_{d+1}=0$. 
	\item There exists an ordering $\{ V^*_i \}_{i=0}^{\delta}$ of the eigenspaces of $A^*$ such that 
	$$
	AV_i^* \subseteq V^*_{i-1}+V^*_i+V^*_{i+1} \qquad (0 \leq i \leq \delta),
	$$ 
	where $V^*_{-1}=V^*_{\delta+1}=0$. 
	\item There does not exist a subspace $W$ of $V$ such that $W \ne 0$, $W \ne V$, $A W \subseteq W$, and $A^* W \subseteq W$. 
\end{enumerate}
We say the pair $A, A^*$ is \emph{over} $\mathbb{F}$.
\end{definition}

\begin{note}\label{note:A*}
It is common to denote the conjugate transpose of $A$ by $A^*$.  
However, we do not use this notation in this paper.  
In a tridiagonal pair $A, A^*$, the $\mathbb{F}$-linear maps $A$ and $A^*$ are arbitrary subject to conditions (i)–(iv) in Definition~\ref{def1}.
\end{note}

Let  $A, A^*$ be a TD pair on $V$ as in Definition~\ref{def1}.  
By \cite[Lemma 4.5]{ITT} the integers $d$ and $\delta$ from {\rm (ii)}, {\rm (iii)} of Definition~\ref{def1} are equal. 
We call this common value the \emph{diameter} of the TD pair. 
An ordering of the eigenspaces of $A$ (resp. $A^*$) is called \emph{standard} whenever it satisfies \rm{(ii)} (resp. \rm{(iii)}) of Definition \ref{def1}.
Let $\{V_i\}_{i=0}^d$ be a standard ordering of the eigenspaces of $A$.
Then the inverted ordering $\{V_{d-i}\}_{i=0}^d$ is also standard, and no further ordering is standard \cite[Lemma 2.4]{ITT}.
A similar result holds for $A^*$.
An ordering of the primitive idempotents of $A$ (resp. $A^*$) is called \emph{standard} whenever the corresponding ordering of the eigenspaces of $A$ (resp. $A^*$) is standard.

\smallskip
When working with a TD pair, it is convenient to use a closely related concept known as a TD system.

\begin{definition}[{\cite[Definition~2.1]{ITT}}]\label{def:TDsystem}
Let $V$ denote a vector space of finite positive dimension. 
By a \emph{tridiagonal system} (or \emph{TD system}) on $V$, we mean a sequence 
\begin{equation}\label{eq: def TDsystem}
	\Phi = ( A; \{ E_i \}_{i=0}^d; A^*; \{ E_i^* \}_{i=0}^{d} )
\end{equation}
that satisfies (i)--(iii) below.
\begin{enumerate}[label=(\roman*),font=\rm,itemsep=0mm]
	\item $A$, $A^*$ is a tridiagonal pair on $V$. 
	\item $\{E_i\}_{i=0}^d$ is a standard ordering of the primitive idempotents of $A$. 
	\item $\{E^*_i\}_{i=0}^d$ is a standard ordering of the primitive idempotents of $A^*$.
\end{enumerate}
\end{definition}

Let $\Phi$ be a TD system on $V$ as in Definition \ref{def:TDsystem}.
For $0 \leq i \leq d$, define 
$$
	\rho_i = \dim(E_i^*V).
$$
By construction, $\rho_i \neq 0$ for all $0 \leq i \leq d$. 
By \cite[Corollaries 5.7, 6.6]{ITT},
\begin{align}
	& \rho_i = \dim(E_iV) && \hspace{-3cm} (0\leq i \leq d), \label{eq(1): shape duality} \\
	& \rho_i = \rho_{d-i} && \hspace{-3cm} (0\leq i \leq d), \label{eq(2): shape sym} \\
	& \rho_{i-1} \leq \rho_i && \hspace{-3cm} (0\leq i \leq d/2). \label{eq(3): shape unim}
\end{align}
The sequence $\{ \rho_i\}_{i=0}^d$ is called the \emph{shape} of $\Phi$. 
By the shape of the TD pair $A, A^*$, we mean the shape of $\Phi$.
We mention a special case of the shape of a TD pair.
A TD pair with shape $(1,1,\ldots, 1)$ is called a \emph{Leonard pair}; see \cite{Ter2001LAA}.

\begin{definition}[{\cite[Definition 1.5]{NT2008LAA2}}]\label{def:sharp}
A TD pair $A, A^*$ is said to be \emph{sharp} whenever $\rho_0 = 1$, where $\{ \rho_i\}^d_{i=0}$ is the shape of $A, A^*$.
\end{definition}

\begin{lemma}[{\cite[Theorem 1.3]{NT2008LAA}}]\label{A,A*:sharp}
A TD pair over an algebraically closed field is sharp.
\end{lemma}

The sharp TD pairs were studied in \cite{NT2008LAA2, NT2008LAA}, and they were classified up to isomorphism \cite{INP2011LAA}.
For more information on TD pairs and TD systems, refer to \cite{ITT, INP2011LAA, IT2004JPAA, IT2007RJ, IT2007LAA, IT2010JAA, NT2007LAA, NT2008LAA2, NT2008LAA}.

\section{Distance-regular graphs}\label{sec:DRGs}
We now turn our attention to distance-regular graphs.
In this section, we review some definitions and preliminaries concerning distance-regular graphs.
In the following sections, we discuss the Bose-Mesner algebra, the dual Bose-Mesner algebra, and the Terwilliger algebra associated with these graphs.
For further details on distance-regular graphs, we refer the reader to \cite{BI1984, BCN, DKT2016, Ter2024}.

\smallskip
Let $\Gamma$ denote a finite, undirected, connected graph, without loops or multiple edges, with vertex set $X$.
Let $\partial$ denote the path-length distance function for $\Gamma$.
Two vertices $y,z \in X$ are said to be \emph{adjacent} whenever $\partial(y,z)=1$.
Define $D = \max\{ \partial(y,z) \mid y,z \in X\}$.
We call $D$ the \emph{diameter} of $\Gamma$.
For $y \in X$ and an integer $i \geq 0$, define 
\begin{equation*}
	\Gamma_i(y) = \{ z \in X \mid \partial(y,z) = i \}.
\end{equation*}
We abbreviate $\Gamma(y) = \Gamma_1(y)$.
For an integer $k \geq 0 $, the graph $\Gamma$ is said to be \emph{regular with valency} $k$ whenever every vertex is adjacent to exactly $k$ vertices.
The graph $\Gamma$ is said to be \emph{distance-regular} whenever, for all integers $0 \leq h, i, j \leq D$ and all vertices $y, z \in X$ with $\partial(y, z) = h$, the number  
\begin{equation}\label{eq:int num}
	p_{i,j}^h = \left| \Gamma_i(y) \cap \Gamma_j(z) \right|  
\end{equation}  
is independent of the choice of $y$ and $z$, and depends only on $h, i, j$.  
The parameter $p^h_{i,j}$ is called an \emph{intersection number} of $\Gamma$.

\smallskip
For the rest of this paper, assume that $\Gamma$ is distance-regular with $D \geq 3$.
By construction, $p^h_{i,j} = p^h_{j,i}$ for all $0 \leq h, i, j \leq D$. 
By the triangle inequality, for $0 \leq h,i,j \leq D$ we have
\begin{enumerate}[label=(\roman*),font=\rm]
\setlength\itemsep{0pt}
	\item $p^h_{i,j} = 0$ if one of $h,i,j$ is greater than the sum of the other two;
	\item $p^h_{i,j} \neq 0$ if one of $h,i,j$ is equal to the sum of the other two.
\end{enumerate}
We abbreviate
\begin{equation}
	c_i = p^i_{1,i-1} \ (1 \leq i \leq D), \qquad 
	a_i = p^i_{1,i} \ (0 \leq i \leq D), \qquad 
	b_i = p^i_{1,i+1} \ (0 \leq i \leq D-1). \label{eq:int num abc}
\end{equation}
Note that $a_0 = 0$, $c_1 = 1$, and 
\begin{equation*}
	c_i > 0 \quad (1 \leq i \leq D), \qquad \qquad b_i > 0 \quad (0 \leq i \leq D-1). 
\end{equation*}
Also, abbreviate
\begin{equation}\label{def:ith_valency}
	k_i = p^0_{i,i} \quad (0 \leq i \leq D).
\end{equation}
Note that $k_0 = 1$ and $k_i = |\Gamma_i(y)|$ for all $y \in X$.
The graph $\Gamma$ is regular with valency $k =k_1= b_0$.
Moreover,
\begin{equation*}
	c_i + a_i + b_i = k \qquad \qquad (0 \leq i \leq D),
\end{equation*}
where $c_0 := 0$ and $b_D := 0$.

\section{The Bose-Mesner algebra}\label{sec:BMalg}
We continue our discussion of the distance-regular graph $\Gamma$. 
In this section, we recall the Bose-Mesner algebra.
For the rest of this paper, we adopt the following conventions.
Assume $\mathbb{F}=\mathbb{R}$ or $\mathbb{F}=\mathbb{C}$.
Let $\Mat_X(\mathbb{F})$ denote the matrix algebra consisting of the matrices with rows and columns indexed by $X$ and all entries in $\mathbb{F}$.
Let $I \in \Mat_X(\mathbb{F})$ denote the identity matrix.
Let $\mathbb{F}^X$ denote the vector space consisting of the column vectors with coordinates indexed by $X$ and all entries in $\mathbb{F}$.  
The algebra $\Mat_X(\mathbb{F})$ acts on $\mathbb{F}^X$ by left multiplication.
We refer to $\mathbb{F}^X$ as the \emph{standard module}. 
For notational convenience, we abbreviate $V = \mathbb{F}^X$.
We endow $V$ with an inner product defined by
\begin{equation*}
	\langle u, v \rangle =
	\begin{cases}
	\overline{u}^\top v 	& \text{ if } \quad  \mathbb{F}=\mathbb{C},  \\
	u^\top v 			& \text{ if } \quad \mathbb{F}=\mathbb{R},
	\end{cases}
\end{equation*}
where $\top$ denotes transpose and $-$ denotes complex conjugation.
Note that $\langle Bu, v \rangle = \langle u, \overline{B}^\top v \rangle$ for $u,v \in V$ and $B \in \Mat_X(\mathbb{F})$.
For $y \in X$, let $\hat{y}$ denote the vector in $V$ with a $1$ in the $y$-coordinate and $0$ in all other coordinates. 
We call $\hat{y}$ the \emph{characteristic vector} of $y$.
Note that the set $\{\hat{y} \mid y \in X\}$ forms an orthonormal basis for $V$.  
We denote $\mathds{1} = \sum_{y \in X} \hat{y}$ and let $\mathbb{J} \in \Mat_X(\mathbb{F})$ denote the matrix that has all entries $1$.

\smallskip
We recall the Bose-Mesner algebra of $\Gamma$.
For $0 \leq i \leq D$, define a matrix $A_i \in \Mat_X(\mathbb{F})$ with $(y,z)$-entry
\begin{equation}\label{def:dis mat}
	(A_i)_{y,z} = 
	\begin{cases} 
		1 & \text{ if } \quad \partial(y, z) = i, \\ 
		0 & \text{ if } \quad  \partial(y, z) \neq i 
	\end{cases} 
	\qquad (y, z \in X).
\end{equation}  
We call $A_i$ the \emph{$i$-th distance matrix} of $\Gamma$.
We observe that
\begin{equation}\label{all-one mat}
	A_0 = I, 
	\qquad\sum_{i=0}^D A_i = \mathbb{J},
	\qquad \overline{A}_i = A_i \ (0 \leq i \leq D),
	\qquad A_i^{\top} = A_i \ (0 \leq i \leq D).
\end{equation}
Moreover, 
\begin{equation}\label{eq(2):Ai properties}
	A_i A_j = \sum_{h=0}^D p^h_{i,j} A_h \qquad (0 \leq i, j \leq D).
\end{equation}
By \eqref{def:dis mat}--\eqref{eq(2):Ai properties}, the matrices $\{A_i\}_{i=0}^D$ form a basis for a commutative subalgebra $M$ of $\Mat_X(\mathbb{F})$.
We call $M$ the \emph{Bose-Mesner algebra} of $\Gamma$.
We abbreviate $A = A_1$ and call this the \emph{adjacency matrix} of $\Gamma$.  
By \cite[Corollary 3.4]{Ter2024}, the matrix $A$ generates $M$.
Moreover, since $A$ is real symmetric, it is diagonalizable and has all real eigenvalues.
Consequently, $A$ has exactly $D+1$ mutually distinct eigenvalues, denoted by $\theta_0, \theta_1, \ldots, \theta_D$.
We call $\{ \theta_i \}_{i=0}^D$ the \emph{eigenvalues of $\Gamma$}.
The valency $k$ is the maximal eigenvalue of $\Gamma$ \cite[Proposition 3.1]{Biggs}.
We set $\theta_0 := k$.
For $0 \leq i \leq D$, let $V_i$ denote the eigenspace of $A$ corresponding to $\theta_i$.
Let $E_i \in \Mat_X(\mathbb{F})$ denote the primitive idempotent of $A$ associated with $V_i$.
By linear algebra, 
\begin{equation}\label{Ei transpose}
	\overline{E}_i = E_i, \qquad
	E_i^\top = E_i \qquad (0\leq i \leq D).
\end{equation}
By \eqref{V:ds}, \eqref{Vi=EiV}, we have
\begin{equation}\label{ods:EiV}
	\qquad V = \sum_{i=0}^D E_iV \qquad \qquad \text{(orthogonal direct sum)}.
\end{equation}
Observe that $E_0V = \operatorname{Span}\{\mathds{1}\}$ and thus 
\begin{equation}\label{E_0=|X|^-1J}
	E_0 = |X|^{-1}\mathbb{J}.
\end{equation}
We call $\{E_i\}_{i=0}^D$ the \emph{primitive idempotents} of $\Gamma$. 
By the comment below \eqref{eq(3):I,E_i}, the matrices $\{E_i\}_{i=0}^D$ form a basis for $M$.
For $0 \leq i \leq D$, let $m_i$ denote the dimension of $E_iV$. 
Observe that $m_i = \operatorname{rank}(E_i)$.

\smallskip
We finish this section with a comment.
Let $\circ$ denote the entrywise multiplication in $\Mat_X(\mathbb{F})$.
Observe that $A_i \circ A_j = \delta_{i,j}A_i$ for $0 \leq i,j \leq D$.
Therefore, $M$ is closed under $\circ$.
Since $\{E_i\}_{i=0}^D$ is a basis for $M$, there exist scalars $q^h_{i,j} \in \mathbb{F}$ $(0 \leq h, i, j \leq D)$ such that
\begin{equation}\label{def:Krein para}
	E_i \circ E_j = |X|^{-1} \sum_{h=0}^D q^h_{i,j} E_h \qquad \qquad (0\leq i, j \leq D).
\end{equation}
The scalars $q^h_{i,j}$ are called the \emph{Krein parameters} of $\Gamma$.
By construction, $q^h_{i,j} = q^h_{j,i}$ for $0 \leq h, i, j \leq D$.
The Krein parameters are real and nonnegative \cite[p.~69]{BI1984}.
They are used to define the $Q$-polynomial property of a distance-regular graph, which will be discussed in Section \ref{sec:Q-poly Property}.

\section{The dual Bose-Mesner algebra}
We continue to discuss the distance-regular graph $\Gamma$. 
In this section, we recall the dual Bose-Mesner algebras of $\Gamma$.
Fix a vertex $x \in X$.
We view $x$ as a ``base vertex.''
For $0 \leq i \leq D$, let $E_i^* = E_i^*(x)$ denote the diagonal matrix in $\Mat_X(\mathbb{F})$ with $(y,y)$-entry
\begin{equation}\label{def:Ei* matrix}
	(E_i^*)_{y,y} =
	\begin{cases}
	1 & \text{ if } \quad \partial(x,y) = i,  \\
	0 & \text{ if } \quad \partial(x,y) \neq i 
	\end{cases}
	\qquad (y \in X).
\end{equation}
We call $E_i^*$ the \emph{$i$-th dual primitive idempotent of $\Gamma$ with respect to $x$}.
Observe that
\begin{align}
	\sum_{i=0}^D E^*_i = I, \qquad \quad
	& \overline{E_i^*} = E_i^* \ (0 \leq i \leq D), \qquad E_i^{*\top} = E_i^* \ (0 \leq i \leq D), \label{eq:dprim (1)}\\
	& E_i^* E_j^* = \delta_{i,j} E_i^* \ (0 \leq i, j \leq D). \label{eq:dprim (2)}
\end{align}
By \eqref{def:Ei* matrix}--\eqref{eq:dprim (2)}, the matrices $\{E_i^*\}_{i=0}^D$ form a basis for a commutative subalgebra $M^* = M^*(x)$ of $\Mat_X(\mathbb{F})$.
We call $M^*$ the \emph{dual Bose-Mesner algebra of $\Gamma$ with respect to $x$} \cite[p.~378]{Ter1992JAC}.  
By \eqref{def:Ei* matrix}, we find that for $0 \leq i \leq D$
\begin{equation}\label{eq:Ei*V}
	E_i^*V = \operatorname{Span} \{\hat{y} \mid y \in X, \ \partial(x,y)=i \}.
\end{equation}
We call $E_i^*V$ the \emph{$i$-th subconstituent of $\Gamma$ with respect to $x$}.
Observe that $\dim(E_i^*V) = k_i$, where $k_i$ is from \eqref{def:ith_valency}.  
Further, observe that $E_0^*V = \mathbb{F}\hat{x}$.  
By \eqref{eq:Ei*V},
\begin{equation}\label{ods:E*iV}
	V = \sum_{i=0}^D E_i^*V \qquad \text{(orthogonal direct sum)}.
\end{equation} 
By construction, the adjacency matrix $A$ satisfies
\begin{equation}\label{eq:A-action}
	AE_i^*V \subseteq E_{i-1}^*V + E_i^*V + E_{i+1}^*V \qquad (0 \leq i \leq D),
\end{equation}
where $E_{-1}^*=0$ and $E_{D+1}^*=0$.

\smallskip
Next, we recall the dual distance matrices of $\Gamma$.
For $0 \leq i \leq D$, let $A_i^* = A_i^*(x)$ denote the diagonal matrix in $\Mat_X(\mathbb{F})$ with $(y,y)$-entry
\begin{equation}\label{def:Ai* matrix}
	(A_i^*)_{y,y} = |X| (E_i)_{x,y} \qquad (y \in X).
\end{equation}
We call $A_i^*$ the \emph{i-th dual distance matrix of $\Gamma$ with respect to $x$}.
By \cite[Lemma 5.9]{Ter2024}, we have
\begin{equation}\label{eq: dual mats}
	A_0^* = I, 
	\qquad\sum_{i=0}^D A_i^* = |X|E_0^*.
	\qquad \overline{A_i^*} = A_i^* \ (0 \leq i \leq D),
	\qquad A_i^{*\top} = A_i^* \ (0 \leq i \leq D).
\end{equation}
Moreover, 
\begin{equation}\label{eq:A*iA*j}
	A_i^*A_j^* = \sum_{h=0}^D q^h_{i,j} A_h^* \qquad (0 \leq i, j \leq D).
\end{equation}
By \cite[Lemma 5.8]{Ter2024}, the matrices $\{A_i^*\}_{i=0}^D$ form a basis for $M^*$.

\section{The Terwilliger algebra}\label{sec:T-alg}
We continue our discussion of the distance-regular graph $\Gamma$.
Recall the base vertex $x \in X$.
Recall the Bose-Mesner algebra $M$ and the dual Bose-Mesner algebra $M^*=M^*(x)$.
Let $T=T(x)$ denote the subalgebra of $\Mat_X(\mathbb{F})$ generated by $M$ and $M^*$.
The algebra $T$ is called the \emph{Terwilliger algebra of $\Gamma$ with respect to $x$} \cite[Definition 3.3]{Ter1992JAC}.
Note that $T$ is finite-dimensional and noncommutative.
By \cite[Lemma~3.2]{Ter1992JAC}, the following relations hold in $T$. 
For $0 \leq h,i,j \leq D$,
\begin{align}
	E_i^*A_hE_j^* = 0 \quad & \text{if and only if} \quad p^h_{i,j} = 0, \label{T rel(1)}\\
	E_iA_h^*E_j = 0 \quad & \text{if and only if} \quad q^h_{i,j} = 0. \label{T rel(2)}
\end{align}
By construction, each of $M$ and $M^*$ is closed under the conjugate-transpose map.
Therefore, $T$ is also closed under the conjugate-transpose map.
In particular, $T$ is semisimple \cite[Lemma~3.4]{Ter1992JAC}.

\smallskip
Next, we recall the $T$-modules.
By a \emph{$T$-module}, we mean a subspace $W$ of $V$ such that $TW \subseteq W$.
For example, $V$ is a $T$-module.
A $T$-module $W$ is \emph{irreducible} whenever $W \ne 0$ and $W$ does not contain a $T$-module other than $0$ and $W$. 
Let $W$ denote a $T$-module and let $U$ denote a $T$-module contained in $W$.
Since $T$ is closed under the conjugate-transpose map, the orthogonal complement of $U$ in $W$ is also a $T$-module.
Consequently, $W$ decomposes into an orthogonal direct sum of irreducible $T$-modules.
In particular, $V$ is an orthogonal direct sum of irreducible $T$-modules.

\begin{example}[{\cite[Lemma 3.6]{Ter1992JAC}}]
Recall the vector $\mathds{1} = \sum_{y \in X} \hat{y}$.
This vector satisfies
\begin{equation}
	A_i\hat{x}=E_i^*\mathds{1}, \qquad
	A_i^*\mathds{1}=|X|E_i\hat{x} \qquad
	(0 \leq i \leq D).
\end{equation}
Moreover, $M\hat{x} = M^*\mathds{1}$.
This common subspace is an irreducible $T$-module.
This $T$-module is called \emph{primary}.
\end{example}

We recall the notion of isomorphism for $T$-modules.
Let $W$ and $W'$ denote $T$-modules.
An $\mathbb{F}$-linear map $\sigma : W \to W'$ is said to be \emph{$T$-linear} whenever 
$$
	\sigma(Bw) = B\sigma(w)
$$
for all $B\in T$ and $w\in W$.
By a \emph{$T$-module isomorphism} from $W$ to $W'$, we mean a bijective $T$-linear map from $W$ to $W'$.
We say that $W$ and $W'$ are \emph{isomorphic} whenever there exists a $T$-module isomorphism from $W$ to $W'$.

\smallskip
Let $W$ denote an irreducible $T$-module.
By \eqref{ods:E*iV}, the subspace $W$ is the direct sum of the nonzero spaces among $\{E_i^*W\}_{i=0}^D$.
By the \emph{endpoint} of $W$, we mean $\min \{i \mid 0 \leq i \leq D, \; E_i^*W \neq 0 \}$.
By the \emph{diameter} of $W$, we mean $ | \{i \mid 0 \leq i \leq D, \; E_i^*W \neq 0 \} | -  1$.
Let $r$ denote the endpoint of $W$ and let $d$ denote the diameter of $W$. 
By \cite[Lemma~3.9]{Ter1992JAC}, we have $r+d \leq D$ and $E_i^*W \neq 0$ if and only if $r \leq i \leq r+d$ $(0 \leq i \leq D)$.
By construction, we have
\begin{equation}\label{eq:A act W}
	AE_i^*W \subseteq E_{i-1}^*W + E_i^*W + E_{i+1}^*W \qquad (r \leq i \leq r+d).
\end{equation}
For more detailed information about the Terwilliger algebra, refer to \cite{Ter1992JAC, Ter1993JAC-1, Ter1993JAC-2}.

\section{The $Q$-polynomial property}\label{sec:Q-poly Property}
We continue to discuss the distance-regular graph $\Gamma$.
In this section, we discuss the $Q$-polynomial property.
Recall the primitive idempotents $\{E_i\}_{i=0}^D$ of $\Gamma$ and the Krein parameters $q^h_{i,j}$ from \eqref{def:Krein para}.
The ordering $\{E_i\}_{i=0}^D$ is called \emph{$Q$-polynomial} whenever the following hold for $0 \leq h, i, j \leq D$:
\begin{enumerate}[label=(\roman*),font=\rm]
\item $q^h_{i,j} = 0$ if one of $h,i,j$ is greater than the sum of the other two;
\item $q^h_{i,j} \neq 0$ if one of $h,i,j$ is equal to the sum of the other two.
\end{enumerate}
We abbreviate 
\begin{equation}\label{eq:dual int num abc}
	c^*_i = q^i_{1,i-1} \ (1 \leq i \leq D), \qquad 
	a^*_i = q^i_{1,i} \ (0 \leq i \leq D), \qquad 
	b^*_i = q^i_{1,i+1} \ (0 \leq i \leq D-1). 
\end{equation}
Note that $a_0^*=0$, $c_1^*=1$, and 
\begin{equation*}
	c_i^* > 0 \quad (1 \leq i \leq D), \qquad \qquad b_i^* > 0 \quad (0 \leq i \leq D-1). 
\end{equation*}
The graph $\Gamma$ is said to be \emph{$Q$-polynomial} whenever there exists at least one $Q$-polynomial ordering of the primitive idempotents.

\smallskip
For the rest of this paper, we assume that the ordering $\{ E_i \}_{i=0}^D$ is $Q$-polynomial.
Fix a vertex $x\in X$ and write $M^*=M^*(x)$. 
Recall the bases $\{A_i^*\}_{i=0}^D$ and $\{E_i^*\}_{i=0}^D$ for $M^*$.
We abbreviate $A^*=A_1^*$ and call this the \emph{dual adjacency matrix of $\Gamma$ with respect to $x$}.
By \cite[Lemma~3.11]{Ter1992JAC}, the matrix $A^*$ generates $M^*$.
Since $\{E_i^*\}_{i=0}^D$ is a basis for $M^*$, there exist scalars $\{ \theta_i^* \}_{i=0}^D$ such that
\begin{equation*}
	A^* = \sum_{i=0}^D \theta_i^* E_i^*.
\end{equation*}
By \cite[Lemma~3.11]{Ter1992JAC}, the scalars $\{\theta_i^*\}_{i=0}^D$ are real and mutually distinct.
For $0 \leq i \leq D$, we have $A^*E_i^*=E_i^*A^*=\theta_i^*E_i^*$, and $E_i^*V$ is the eigenspace of $A^*$ corresponding to the eigenvalue $\theta_i^*$.
We call $\{ \theta_i^* \}_{i=0}^D$ the \emph{dual eigenvalues of $\Gamma$}.
Consider the Terwilliger algebra $T=T(x)$.
Note that $A$, $A^*$ generates $T$.
Recall the relations \eqref{T rel(1)}, \eqref{T rel(2)} for $T$.
Setting $h=1$ in these relations, we find that for $0 \leq i, j \leq D$,
\begin{align}
	E_i^*AE_j^* = 0 \qquad \text{if} \quad | i - j | > 1, \label{q-T rel(1)}\\
	E_iA^*E_j = 0 \qquad \text{if} \quad | i - j | > 1. \label{q-T rel(2)}
\end{align}
Moreover, setting $j=i$ in \eqref{T rel(1)}, \eqref{T rel(2)}, we find that for $0 \leq h, i \leq D$,
\begin{align}
	E_i^*A_hE_i^* = 0 \quad & \text{if} \quad h > 2i, \label{T rel(1): h >2i}\\
	E_iA_h^*E_i = 0 \quad & \text{if} \quad h>2i. \label{T rel(2): h>2i}
\end{align}
We mention a result for later use.
\begin{lemma}\label{lem: E*AE*sym}
For $0 \leq i \leq D$, consider the subspaces $E_i^* M E_i^*$ and $E_i M^* E_i$.
Then the following {\rm(i)}, {\rm(ii)} hold.
\begin{itemize}
\setlength{\itemsep}{0pt}
	\item[\rm(i)] Every element of $E_i^*ME_i^*$ is symmetric.
	\item[\rm(ii)] Every element of $E_i M^* E_i$ is symmetric.
\end{itemize}
\end{lemma}

\begin{proof}
(i) Since $\{A_j\}_{j=0}^D$ forms a basis for $M$, it suffices to show that $E_i^*A_jE_i^*$ is symmetric for $0 \leq j \leq D$. 
This follows since both $E_i^*$ and $A_j$ are symmetric.\\
(ii) The proof is similar to (i), using the basis $\{A_j^*\}_{j=0}^D$ for $M^*$ and the fact that both $E_i$ and $A_j^*$ are symmetric.
\end{proof}

\section{Action of $A$, $A^*$ on an irreducible $T$-module}\label{sec:T-mod TDpairs}
Recall the $Q$-polynomial distance-regular graph $\Gamma$.
Fix a vertex $x \in X$ and write $T = T(x)$.
In this section, we discuss how the matrices $A$, $A^*$ act on an irreducible $T$-module.
Recall the eigenspaces $\{E_iV\}_{i=0}^D$ of $A$ and the eigenspaces $\{E^*_iV\}_{i=0}^D$ of $A^*$.
In \eqref{eq:A-action}, we mentioned the action of $A$ on $\{E^*_iV\}_{i=0}^D$. 
Using \eqref{q-T rel(2)}, the action of $A^*$ on $\{E_iV\}_{i=0}^D$ is given as follows:
\begin{equation}\label{eq:A*-action}
	A^*E_iV \subseteq E_{i-1}V + E_iV + E_{i+1}V \qquad (0 \leq i \leq D),
\end{equation}
where $E_{-1}=0$ and $E_{D+1}=0$.

\smallskip
Let $W$ denote an irreducible $T$-module. 
By \eqref{ods:EiV}, the subspace $W$ is the direct sum of the nonzero spaces among $\{E_iW\}_{i=0}^D$.
By the \emph{dual endpoint} of $W$, we mean $\min \{i \mid 0 \leq i \leq D, \; E_iW \neq 0 \}$.
By the \emph{dual diameter} of $W$, we mean $| \{i \mid 0 \leq i \leq D, \; E_iW \neq 0 \} | -  1$.
Let $s$ denote the dual endpoint of $W$ and let $\delta$ denote the dual diameter of $W$.
By \cite[Lemma~3.12]{Ter1992JAC}, we have $s+\delta \leq D$ and $E_iW \neq 0$ if and only if $s \leq i \leq s+\delta$ $(0 \leq i \leq D)$.
Moreover, the dual diameter of $W$ is equal to the diameter of $W$ \cite[Corollary~13.7]{Ter2024}.
By construction, we have
\begin{equation}\label{eq:A* act W}
	A^*E_iW \subseteq E_{i-1}W + E_iW + E_{i+1}W \qquad (s \leq i \leq s+d).
\end{equation}

\begin{lemma}[{\cite[Lemmas~3.9,~3.12]{Ter1992JAC}}]\label{A,A*;TDpair}
Let $W$ denote an irreducible $T$-module.
Then the pair $A$, $A^*$ acts on $W$ as a TD pair.
\end{lemma}

\begin{lemma}[{\cite[Lemmas~3.9,~3.12]{Ter1992JAC}}]\label{A,A*;TDsystem}
Let $W$ denote an irreducible $T$-module with endpoint $r$, dual endpoint $s$, and diameter $d$.
Then the sequence
\begin{equation}\label{eq: seq TDsys}
	\Phi=(A; \{E_{s+i}\}_{i=0}^{d}; A^*; \{E_{r+i}^*\}_{i=0}^{d})
\end{equation}
acts on $W$ as a TD system.
\end{lemma}

Let $W$ denote an irreducible $T$-module with endpoint $r$, dual endpoint $s$, and diameter $d$.
Define a sequence $\{\rho_i\}_{i=0}^d$ by
\begin{equation}\label{def:shape}
	\rho_i = \dim(E_{r+i}^*W) \qquad (0 \leq i \leq d).
\end{equation}
Note that $\rho_i \ne 0$ for $0 \leq i \leq d$.
The sequence $\{\rho_i\}_{i=0}^d$ is called the \emph{shape} of $W$.
The $T$-module $W$ is called \emph{sharp} whenever $\rho_0=1$.
In the next result, we emphasize a point for later use.

\begin{lemma}[{\cite[Corollary~5.7]{ITT}}]\label{lem: dim sym}
Let $W$ denote an irreducible $T$-module with endpoint $r$, dual endpoint $s$, and diameter $d$.
For $0 \leq i \leq d$, the following subspaces have the same dimension:
\begin{equation}\label{eq: 4 subspaces}
	E_{r+i}^*W, \qquad E_{r+d-i}^*W, \qquad E_{s+i}W, \qquad E_{s+d-i}W.
\end{equation}
\end{lemma}

We finish this section with some comments on the primary module.

\begin{lemma}[{cf.~\cite[Lemma~3.6]{Ter1992JAC},~\cite[Proposition~8.3]{Egge2000}}]\label{lem:equi pri Tmod}
Let $W$ denote an irreducible $T$-module with endpoint $r$, dual endpoint $s$, and diameter $d$.
Then the following {\rm(i)}--{\rm(iv)} are equivalent.
\begin{itemize}
\setlength{\itemsep}{0pt}
	\item[\rm(i)] $W$ is primary.
	\item[\rm(ii)] $r=0$.
	\item[\rm(iii)] $s=0$.
	\item[\rm(iv)] $d=D$.
\end{itemize}
Suppose that $W$ is primary. Then
\begin{equation}\label{eq:Ei*W=1}
	\rho_i = 1 \qquad  (0 \leq i \leq d), 
\end{equation}	
where $\{ \rho_i \}_{i=0}^d$ is the shape of $W$.
\end{lemma}
\begin{proof}
The result follows from \cite[Lemma~3.6]{Ter1992JAC}.
\end{proof}

\section{Proof of Theorem \ref{mainthm}}\label{sec:pf thm}

In this section, we prove Theorem \ref{mainthm}. 
Recall the $Q$-polynomial distance-regular graph $\Gamma$.
Fix $x \in X$ and write $T=T(x)$.
We have preliminary comments.
For $0 \leq i \leq D$, consider the subspace $E_i^*TE_i^*$. 
Note that $E_i^*TE_i^*$ is an algebra with multiplicative identity $E_i^*$.
By construction, $E_i^*TE_i^*$ is finite-dimensional and closed under the conjugate transpose map, hence it is semisimple \cite{CR1962}. 
In general, this algebra is not commutative \cite[Theorem 5.1]{Ter1993JAC-2}. 
Note that the $i$-th subconstituent $E_i^*V$ is invariant under the action of $E_i^*TE_i^*$.

\begin{lemma}[{\cite[Theorem~2.6]{NT2008LAA}}]\label{NT:Thm2.6}
Let $W$ denote an irreducible $T$-module with endpoint $r$, dual endpoint $s$, and diameter $d$.
Then the following {\rm(i)}--{\rm(iii)} hold on $W$.
\begin{itemize}
\setlength{\itemsep}{0pt}
	\item[\rm(i)] The algebra $E_r^* T E_r^*$ is generated by $E_r^* M E_r^*$.
	\item[\rm(ii)] The elements of $E_r^* M E_r^*$ mutually commute.
	\item[\rm(iii)] The algebra $E_r^* T E_r^*$ is commutative.
\end{itemize}
\end{lemma}
\begin{proof} 
This follows from \cite[Theorem 2.6]{NT2008LAA}.
\end{proof}

\noindent
For $0 \leq i \leq D$, consider the subspace $E_i T E_i$. 
Note that $E_i T E_i$ is an algebra with the multiplicative identity $E_i$.
By construction, $E_i T E_i$ is finite-dimensional and closed under the conjugate transpose map, hence it is semisimple \cite{CR1962}.
In general, this algebra is not commutative \cite[Theorem 5.1]{Ter1993JAC-2}. 
Note that the subspace $E_iV$ is invariant under the action of $E_i T E_i$.

\begin{lemma}[{\cite[Theorem~2.6]{NT2008LAA}}]\label{NT:Thm2.6 dual}
Let $W$ denote an irreducible $T$-module with endpoint $r$, dual endpoint $s$, and diameter $d$.
Then the following {\rm(i)}--{\rm(iii)} hold on $W$.
\begin{itemize}
\setlength{\itemsep}{0pt}
	\item[\rm(i)] The algebra $E_s T E_s$ is generated by $E_s M^* E_s$.
	\item[\rm(ii)] The elements of $E_s M^* E_s$ mutually commute.
	\item[\rm(iii)] The algebra $E_s T E_s$ is commutative.
\end{itemize}
\end{lemma}
\begin{proof} 
This follows from \cite[Theorem 2.6]{NT2008LAA}.
\end{proof}

We are now ready to prove Theorem \ref{mainthm}.

\begin{proof}[Proof of Theorem \ref{mainthm}]
Assume $\mathbb{F}=\mathbb{R}$.
Let $W$ denote an irreducible $T$-module with endpoint $r$, dual endpoint $s$ and diameter $d$.
Let $\{\rho_i\}_{i=0}^d$ denote the shape of $W$.
We will show that $\rho_0=1$; that is, $\dim(E_r^*W)=1$.
Observe that $E_r^*W$ is non-zero, finite-dimensional, and invariant under $E_r^* T E_r^*$.
Consider the restriction of $E_r^* T E_r^*$ to $W$.
By Lemmas \ref{lem: E*AE*sym}(i) and \ref{NT:Thm2.6}, all elements of $E_r^*TE_r^*$ are simultaneously diagonalizable on $W$.
Therefore, there exists a nonzero vector $w \in E_r^*W$ that is a common eigenvector for $E_r^* T E_r^*$.

Now, consider the subspace $Tw$ of $V$.
Observe that $Tw$ is a nonzero $T$-submodule of $W$.
By the irreducibility of $W$, we have $Tw=W$.
Using this and the fact that $w \in E_r^*W$, applying $E_r^* T E_r^*$ to $w$ gives
$$
	E_r^* T E_r^* w = E_r^* T w=E_r^* W.
$$
Since $w$ is a common eigenvector for $E_r^* T E_r^*$, we have
$
	E_r^* T E_r^* w = \operatorname{Span}\{ w\},
$ 
which is equal to $E_r^* W$.
Therefore, $\dim(E_r^*W) = 1$, and hence $\rho_0 = 1$.
The proof is complete.
\end{proof}

Next, we prove Corollary \ref{main cor}.
\begin{proof}[Proof of Corollary \ref{main cor}]
Consider a $Q$-polynomial distance-regular graph with diameter $D$.
If $D=1$, then the graph is complete, which is trivially sharp.
If $D=2$, then the graph is strongly regular and sharp by the result of \cite{TY1994}.
If $D \geq 3$, then the graph is sharp by Theorem~\ref{mainthm}.
The result follows.
\end{proof}

\section{The complexification of a $T$-module}\label{sec:pf thm2}
In this section, we prove Theorems~\ref{thm:complexification},~\ref{thm:WC iso}, and~\ref{thm:Wedderburn}.
Recall the $Q$-polynomial distance-regular graph $\Gamma$.  
Fix $x \in X$ and write $T = T(x)$.
Recall the ground field $\mathbb{F}=\mathbb{R}$ or $\mathbb{F}=\mathbb{C}$.
When it is important to emphasize the ground field, we will write $T=T^\mathbb{F}$.
By construction, $T^\mathbb{F}$ is the subalgebra of $\Mat_X(\mathbb{F})$ generated by $A$ and $A^*$.
We now compare $T^\mathbb{R}$ and $T^\mathbb{C}$.
Recall that $A, A^* \in \Mat_X(\mathbb{R})$.
By construction,
\begin{equation}\label{eq:TC=TR+iTR}
	T^\mathbb{C} = T^\mathbb{R} + \boldsymbol{i}  T^\mathbb{R},
\end{equation}
where $\boldsymbol{i} = \sqrt{-1}$ and the sum in \eqref{eq:TC=TR+iTR} is direct over $\mathbb{R}$.
Note that $\dim_{\mathbb{R}}T^\mathbb{R} = \dim_{\mathbb{C}} T^\mathbb{C}$.
By \eqref{eq:TC=TR+iTR}, we have
\begin{equation}\label{eq:T^C}
	T^\mathbb{C} = T^\mathbb{R} \otimes \mathbb{C},
\end{equation}
where we understand $\otimes=\otimes_{\mathbb{R}}$.
Motivated by \eqref{eq:T^C}, we call $T^\mathbb{C}$ the \emph{complexification of $T^\mathbb{R}$}.

\smallskip
Next, let $W$ denote a $T^\mathbb{R}$-module.  
The \emph{complexification of $W$} is the $\mathbb{C}$-vector space
\begin{equation}\label{def:complexification}
	W^\mathbb{C} = W \otimes \mathbb{C},
\end{equation}
see \cite[p.~379]{Roman2010}.
Then $W^\mathbb{C}$ is a $T^\mathbb{C}$-module.
By construction,
\begin{equation}\label{W_C:C-v.s}
    W^\mathbb{C} = W + \boldsymbol{i}W,
\end{equation}
where the sum in \eqref{W_C:C-v.s} is direct over $\mathbb{R}$.
Note that $\dim_{\mathbb{R}}W = \dim_\mathbb{C} W^\mathbb{C}$.

\smallskip
We are now ready to prove Theorem \ref{thm:complexification}.

\begin{proof}[Proof of Theorem \ref{thm:complexification}]
Recall that $W$ is an irreducible $T^\mathbb{R}$-module.
Let $r$ denote the endpoint of $W$.
By Theorem~\ref{mainthm}, $W$ is sharp, so  $\dim_{\mathbb{R}}(E_r^*W)=1$.
Consider the complexification $W^\mathbb{C}$ of $W$.
By construction, $\dim_{\mathbb{C}}(E_r^*W^\mathbb{C})=1$.
Decompose $W^\mathbb{C}$ as a direct sum of irreducible $T^\mathbb{C}$-modules to get
\begin{equation}\label{eq:dsWC}
	W^\mathbb{C} = \sum_j W_j.
\end{equation}
Applying $E_r^*$ to \eqref{eq:dsWC}, we obtain a direct sum of $\mathbb{C}$-vector spaces
\begin{equation}\label{eq:dsEr*WC}
	E_r^*W^\mathbb{C} = \sum_j E_r^*W_j.
\end{equation}
Since the left-hand side of \eqref{eq:dsEr*WC} has dimension $1$, there is a unique nonzero summand on the right-hand side of \eqref{eq:dsEr*WC}.
We denote this summand by $E_r^*W_{\xi}$.
It follows that $E_r^*W^\mathbb{C} = E_r^*W_{\xi} \subseteq W_{\xi}$.
Moreover, $E_r^*W^\mathbb{C}$ is the $\mathbb{C}$-span of $E_r^*W$, so 
$$
	E_r^*W \subseteq E_r^*W^\mathbb{C} \subseteq W_{\xi}.
$$
Pick $0 \ne z \in E_r^*W \subseteq W_{\xi}$.
Since $W_{\xi}$ is an irreducible $T^\mathbb{C}$-module, we have $W_{\xi} = T^\mathbb{C} z$.
Also, by the irreducibility of the $T^\mathbb{R}$-module $W$, we have $W = T^\mathbb{R} z$.
Since $T^\mathbb{R} z \subseteq T^\mathbb{C} z$, it follows that $W  \subseteq  W_\xi$.
Furthermore, $\boldsymbol{i} W \subseteq W_{\xi}$ since $W_{\xi}$ is closed under multiplication by complex scalars.
Therefore,
\begin{equation}\label{eq:W+iW}
	W + \boldsymbol{i} W  \subseteq W_{\xi}.
\end{equation}
By \eqref{W_C:C-v.s} and \eqref{eq:W+iW}, we obtain $W^\mathbb{C} \subseteq W_{\xi}$.
Since $W_{\xi}$ is an irreducible $T^\mathbb{C}$-module, it follows that $W_{\xi} = W^\mathbb{C}$.
Hence, the $T^\mathbb{C}$-module $W^\mathbb{C}$ is irreducible. 
\end{proof}

We now prove Theorem \ref{thm:WC iso}.

\begin{proof}[Proof of Theorem \ref{thm:WC iso}]
Recall that $W_1$ and $W_2$ are irreducible $T^\mathbb{R}$-modules.  
Suppose that $W_1$ and $W_2$ are isomorphic as $T^\mathbb{R}$-modules.  
By \eqref{def:complexification}, we have an isomorphism of $T^\mathbb{C}$-modules
$$
	W_1^\mathbb{C} 
	= W_1 \otimes \mathbb{C}
	\;\cong\;  W_2 \otimes \mathbb{C}
	= W_2^\mathbb{C}.
$$

\smallskip
Conversely, suppose $W_1^\mathbb{C}$ and $W_2^\mathbb{C}$ are isomorphic as $T^\mathbb{C}$-modules.  
Let $\eta: W_1^\mathbb{C} \to W_2^\mathbb{C}$ be a $T^\mathbb{C}$-module isomorphism.
For $j \in \{1,2\}$ view $W_j^\mathbb{C} = W_j + \boldsymbol{i} W_j$ as in \eqref{W_C:C-v.s}.
For $w \in W_1 \subseteq W_1^\mathbb{C}$, write
\begin{equation}
	\eta (w) = \sigma_1(w) + \boldsymbol{i} \sigma_2(w),
\end{equation}
where $\sigma_1$, $\sigma_2: W_1 \to W_2$ are $\mathbb{R}$-linear maps. 

\smallskip
We first show that both $\sigma_1$ and $\sigma_2$ are $T^\mathbb{R}$-linear.
Pick $B \in T^\mathbb{R}$ and $w \in W_1$. 
Since $B \in T^\mathbb{C}$ and $\eta$ is $T^{\mathbb{C}}$-linear, we have $\eta(Bw)=B\eta(w)$.
Thus,
$$
	\sigma_1(Bw) + \boldsymbol{i} \sigma_2(Bw) = B \sigma_1(w) + \boldsymbol{i} B \sigma_2(w).
$$
Comparing real and imaginary parts, we obtain $\sigma_1(Bw) = B \sigma_1(w)$ and $\sigma_2(Bw) = B \sigma_2(w)$.
Therefore, both $\sigma_1$ and $\sigma_2$ are $T^\mathbb{R}$-linear.

\smallskip
Next, we show that at least one of $\sigma_1, \sigma_2$ is nonzero.
If $\sigma_1 = \sigma_2 = 0$, then $\eta$ vanishes on $W_1$.
By $\mathbb{C}$-linearity, $\eta(\boldsymbol{i} w) = \boldsymbol{i} \eta(w) = 0$ for all $w \in W_1$, so $\eta$ also vanishes on $\boldsymbol{i} W_1$.
Hence, $\eta = 0$ on $W_1^\mathbb{C} = W_1 + \boldsymbol{i} W_1$.
This contradicts that $\eta$ is an isomorphism. 
Therefore at least one of $\sigma_1,\sigma_2$ is nonzero.
Suppose that $\sigma_1 \neq 0$ (the case $\sigma_2 \ne 0$ is similar).
Then $\sigma_1(W_1)$ is a nonzero submodule of the $T^\mathbb{R}$-module $W_2$, and hence $\sigma_1(W_1) = W_2$ by the irreducibility of the $T^\mathbb{R}$-module $W_2$.
Moreover, $\ker\sigma_1$ is a submodule of the $T^\mathbb{R}$-module $W_1$.  
Since the $T^\mathbb{R}$-module $W_1$ is irreducible and $\sigma_1 \neq 0$, we have $\ker\sigma_1=0$.  
Therefore $\sigma_1: W_1 \to W_2$ is bijective, and hence a $T^\mathbb{R}$-module isomorphism.  
The result follows.
\end{proof}

Next, we prove Theorem \ref{thm:Wedderburn}.
We have some preliminary comments.
For an irreducible $T^{\mathbb{R}}$-module $W$, let $\End_{T^{\mathbb{R}}}(W)$ denote the $\mathbb{R}$-algebra consisting of the $T^{\mathbb{R}}$-linear maps from $W$ to $W$.
Then $\End_{T^{\mathbb{R}}}(W)$ is a finite-dimensional division algebra over $\mathbb{R}$ (cf. \cite[\S1.2]{FH1991}).
Consequently, $\End_{T^{\mathbb{R}}}(W)$ is isomorphic to $\mathbb{R}$, $\mathbb{C}$, or $\mathbb{H}$ as an $\mathbb{R}$-algebra (cf. \cite[Theorem~13.12]{Lam2001}).

\begin{lemma}\label{lem:End_T(W)=R}
Assume that $\mathbb{F}=\mathbb{R}$.
Let $W$ denote an irreducible $T$-module.
Then $\End_T(W)$ is isomorphic to $\mathbb{R}$ as an $\mathbb{R}$-algebra.
\end{lemma}
\begin{proof}
Let $r$ denote the endpoint of $W$.
Set $U_0 = E_r^*W$.
Since $U_0 \neq 0$ and $W$ is irreducible, we have $W = T U_0$.
Let $\End (U_0)$ denote the $\mathbb{R}$-algebra consisting of the $\mathbb{R}$-linear maps from $U_0$ to $U_0$.
Since $\dim U_0 = 1$ by Theorem \ref{mainthm}, every element of $\End (U_0)$ acts on $U_0$ as multiplication by a real scalar.
Hence the $\mathbb{R}$-algebra $\End (U_0)$ is isomorphic to $\mathbb{R}$.

\smallskip
Next, pick $\sigma \in \End_T(W)$.
Since $\sigma$ is $T$-linear, we have
$$
	\sigma(U_0) = \sigma(E_r^*W) = E_r^*\sigma(W) \subseteq E_r^*W=U_0.
$$ 
Thus the restriction map
$$
	\varphi:\End_T(W) \longrightarrow \End (U_0),\qquad
	\varphi(\sigma) = \sigma|_{U_0}
$$
is well-defined.
We claim that $\varphi$ is an $\mathbb{R}$-algebra isomorphism.
To prove this claim, we first show that $\varphi$ is injective.
Suppose $\varphi(\sigma)=0$.
Then $\sigma$ vanishes on $U_0$.
Since $W=TU_0$, for any $w\in W$ we can write $w=Bu$ with some $B\in T$ and $u\in U_0$.
By $T$-linearity,
$$
	\sigma(w) = \sigma(Bu) = B\,\sigma(u) = B \cdot 0 = 0.
$$
Thus $\sigma=0$ on $W$, and hence $\varphi$ is injective.
Next, we show that $\varphi$ is surjective.
Since the identity map $\operatorname{id}_W$ on $W$ belongs to $\End_T(W)$, we have $\varphi(\operatorname{id}_W) = \operatorname{id}_{U_0} = 1 \in \mathbb{R}$.
Thus the image of $\End_T(W)$ under $\varphi$ is an $\mathbb{R}$-subalgebra of $\mathbb{R}$ containing $1$, which must be all of $\mathbb{R}$.
Therefore, $\varphi$ is surjective.
By the above comments, $\varphi$ is an $\mathbb{R}$-algebra isomorphism.
The result follows.
\end{proof}

We are now ready to prove Theorem \ref{thm:Wedderburn}.

\begin{proof}[Proof of Theorem \ref{thm:Wedderburn}]
We first show that $\mathbb{D}_i = \mathbb{R}$ for $1 \leq i \leq \ell$.
Recall the Wedderburn decomposition of $T^\mathbb{R}$ from \eqref{eq:Wed decomp over R}.
For $1 \leq i \leq \ell$, we abbreviate $S_i = \Mat_{m_i}(\mathbb{D}_i)$, which is the $i$-th component of the direct sum in \eqref{eq:Wed decomp over R}.
Let $W_i$ denote an irreducible $T^\mathbb{R}$-module corresponding to $S_i$.
Let $\End_{S_i}(W_i)$ denote the $\mathbb{R}$-algebra consisting of the $S_i$-linear maps from $W_i \to W_i$.
Since every simple component of $T^\mathbb{R}$ other than $S_i$ acts as zero on $W_i$, the $T^\mathbb{R}$-action on $W_i$ coincides with the $S_i$-action on $W_i$. 
It follows that $\End_{T^\mathbb{R}}(W_i) = \End_{S_i}(W_i)$.
By \cite[Theorem~3.3]{Lam2001}, the $\mathbb{R}$-algebra $\End_{S_i} (W_i)$ is isomorphic to $\mathbb{D}_i$.
Therefore, we have an $\mathbb{R}$-algebra isomorphism
\begin{equation}\label{EndT(Wi)=Di}
	\End_{T^\mathbb{R}}(W_i) \cong \mathbb{D}_i.
\end{equation}
On the other hand, by Lemma \ref{lem:End_T(W)=R} we have an $\mathbb{R}$-algebra isomorphism 
\begin{equation}\label{EndT(Wi)=R}
	\End_{T^\mathbb{R}}(W_i) \cong \mathbb{R}.
\end{equation}
By \eqref{EndT(Wi)=Di}, \eqref{EndT(Wi)=R}, we have $\mathbb{D}_i = \mathbb{R}$ for  $1 \leq i \leq \ell$.

\smallskip
Next, we show that $\ell = h$ and $m_i = n_i$ for $1 \leq i \leq \ell$.
We complexify \eqref{eq:Wed decomp over R} to obtain $\mathbb{C}$-algebra isomorphisms
\begin{align*}
	T^\mathbb{C} 
	= T^\mathbb{R} \otimes \mathbb{C}
	& \cong  \left(\bigoplus_{i=1}^\ell \Mat_{m_i}(\mathbb{R})\right) \otimes \mathbb{C} \\ 
	& = \bigoplus_{i=1}^\ell \left( \Mat_{m_i}(\mathbb{R}) \otimes \mathbb{C} \right) \\
	& \cong \bigoplus_{i=1}^\ell \Mat_{m_i}(\mathbb{C}).
\end{align*}
Comparing this with \eqref{eq:Wed decomp over C}, we obtain $\ell = h$ and $m_i = n_i$ for $1 \leq i \leq \ell$.
The proof is complete.
\end{proof}

We finish this section with some consequences of Theorems~\ref{thm:complexification} and \ref{thm:WC iso}.

\begin{lemma}\label{lem:WC structure}
Let $W$ denote an irreducible $T^\mathbb{R}$-module.  
Then the irreducible $T^\mathbb{C}$-module $W^\mathbb{C}$ has the same endpoint, dual endpoint, diameter, and shape as $W$.  
\end{lemma}
\begin{proof} 
For $0 \leq j \leq D$, we have $E_j^*W^\mathbb{C} = E_j^*W + \boldsymbol{i} E_j^*W$. 
Thus, the subspace $E_j^*W^\mathbb{C}$ is the $\mathbb{C}$-span of $E_j^*W$.
By this comment, we have
$$
	\dim_{\mathbb{R}}(E_j^*W) = \dim_{\mathbb{C}}(E_j^*W^\mathbb{C}).
$$
In particular,
$$
	E_j^*W = 0 \quad  \text{if and only if} \quad E_j^*W^\mathbb{C} = 0.
$$
Similarly, we have
$$
	\dim_{\mathbb{R}}(E_jW) = \dim_{\mathbb{C}}(E_jW^\mathbb{C}),
$$
and
$$
	E_jW = 0 \quad  \text{if and only if} \quad E_jW^\mathbb{C} = 0.
$$
The result follows from these comments along with the definitions of the endpoint, dual endpoint, diameter, and shape of $W$.
\end{proof}

\begin{corollary}\label{cor:decomp_preserved}
Let $W$ denote a $T^{\mathbb{R}}$-module.
Consider a direct sum decomposition of $W$ into irreducible $T^\mathbb{R}$-modules
\begin{equation}\label{eq:W=ds W_i over R}
	W = \sum_i W_i.
\end{equation}
Then the complexification $W^\mathbb{C}$ has a direct sum decomposition 
\begin{equation}\label{eq:decomp W C(2)}
   W^{\mathbb{C}} = \sum_i W_i^{\mathbb{C}}.
\end{equation}
Moreover, the following {\rm(i)}, {\rm(ii)} hold.
\begin{enumerate}[label=(\roman*),font=\rm]
\setlength\itemsep{0pt}
	\item Each $W_i^{\mathbb{C}}$ is an irreducible $T^{\mathbb{C}}$-module with the same endpoint, dual endpoint, diameter, and shape as $W_i$.  
	\item The multiplicity of $W_i^{\mathbb{C}}$ in $W^{\mathbb{C}}$ equals the multiplicity of $W_i$ in $W$.  
\end{enumerate}
\end{corollary}

\begin{proof}
This follows from Theorem~\ref{thm:complexification}, Theorem~\ref{thm:WC iso} and Lemma~\ref{lem:WC structure}.
\end{proof}

\section{The algebras $E_1^*TE_1^*$ and $E_1TE_1$}\label{sec:E*1TE*1}

Recall the $Q$-polynomial distance-regular graph $\Gamma$.
Recall $\mathbb{F} = \mathbb{R}$ or $\mathbb{F} = \mathbb{C}$.
Fix $x\in X$ and write $T=T(x)$.
In this section, we discuss the algebras $E_1^* T E_1^*$ and $E_1 T E_1$.
Let $W$ denote an irreducible $T$-module.
Consider the subspaces $E_1^*W$ and $E_1W$.
If $W$ is primary, then both $E_1^*W$ and $E_1W$ have dimension one, as discussed in \eqref{eq:Ei*W=1}.
In the next lemma, we assume that $W$ is nonprimary.

\begin{lemma}\label{lem:dimE1W=1}
Let $W$ denote a nonprimary irreducible $T$-module with endpoint $r$, dual endpoint $s$, and diameter $d$.
Then $1 \leq r, s \leq D$.
Moreover, the following {\rm(i)}--{\rm(iv)} hold.
\begin{itemize}
\setlength\itemsep{0pt}
	\item[{\rm (i)}] If $r=1$, then $E_1^*W$ has dimension one.
	\item[{\rm (ii)}] If $2 \leq r \leq D$, then $E_1^*W=0$.
	\item[{\rm (iii)}] If $s=1$, then $E_1 W$ has dimension one.
	\item[{\rm (iv)}] If $2 \leq s \leq D$, then $E_1 W=0$.	
\end{itemize}
\end{lemma}
\begin{proof}
We have $1 \leq r, s \leq D$ by Lemma \ref{lem:equi pri Tmod}. \\
(i): This follows from Theorem \ref{mainthm}.\\
(iii): This follows from (i) and Lemma \ref{lem: dim sym}\\
(ii),(iv): These follow from the definitions of endpoint and dual endpoint.
\end{proof}

\begin{proposition}\label{lem:E1*TE1* comm}
The algebras $E_1^* T E_1^*$ and $E_1 T E_1$ are commutative.
\end{proposition}
\begin{proof}
We first show that $E_1^* T E_1^*$ is commutative.  
Since $V$ is an orthogonal direct sum of irreducible $T$-modules, it suffices to show that the action of $E_1^* T E_1^*$ on any irreducible $T$-module is commutative.
Let $W$ denote an irreducible $T$-module.
By \eqref{eq:Ei*W=1} and Lemma \ref{lem:dimE1W=1}(i),(ii) either $E_1^*W$ has dimension one or $E_1^*W = 0$.
In either case, the action of $E_1^* T E_1^*$ on $W$ is commutative.
We have shown that $E_1^* T E_1^*$ is commutative.
The argument for $E_1 T E_1$ is similar.
\end{proof}

Recall the all-ones matrix $\mathbb{J}$ from Section \ref{sec:BMalg}.

\begin{lemma}\label{lem:E1*ME1*,E1TE1}
The following {\rm(i)}, {\rm(ii)} hold.
\begin{itemize}
\setlength\itemsep{0pt}
	\item[\rm (i)] $E_1^* M E_1^*$ is spanned by $E_1^*, E_1^* \mathbb{J} E_1^*, E_1^* A E_1^*$.
 	\item[\rm (ii)] $E_1 M E_1$ is spanned by $E_1, E_1 E_0^* E_1, E_1 A^* E_1$.
\end{itemize}
\end{lemma}
\begin{proof}
(i): Since $M$ has a basis $\{A_i\}_{i=0}^D$, the subspace $E_1^* M E_1^*$ is spanned by $\{E_1^* A_i E_1^*\}_{i=0}^D$.
Applying \eqref{T rel(1): h >2i} to $\{E_1^* A_i E_1^*\}_{i=0}^D$, we find that $E_1^* A_i E_1^* = 0$ for $3 \leq i \leq D$. 
Therefore, $E_1^* M E_1^*$ is spanned by $E_1^*$, $E_1^* A E_1^*$ and $E_1^* A_2 E_1^*$. 
Recall that $\mathbb{J} = \sum_{i=0}^D A_i$.
Using this and \eqref{T rel(1): h >2i}, we have
\begin{equation*}
	E_1^*\mathbb{J}E_1^* = E_1^* + E_1^* A E_1^* + E_1^* A_2 E_1^*.
\end{equation*}
From this equation, $E_1^* A_2 E_1^*$ can be expressed in terms of $E_1^*$, $E_1^*\mathbb{J}E_1^*$ and $E_1^* A E_1^*$.
The result follows.\\
(ii): Since $M^*$ has a basis $\{A_i^*\}_{i=0}^D$, the subspace $E_1 M^* E_1$ is spanned by $\{E_1 A_i^* E_1\}_{i=0}^D$.
Applying \eqref{T rel(2): h>2i} to $\{E_1 A_i^* E_1\}_{i=0}^D$, we find that $E_1 A_i^* E_1 = 0$ for $3 \leq i \leq D$. 
Therefore, $E_1 M^* E_1$ is spanned by $E_1$, $E_1 A^* E_1$ and $E_1 A_2^* E_1$. 
Recall that $|X|E_0^* = \sum_{i=0}^D A_i^*$.
Using this and \eqref{T rel(2): h>2i}, we have
\begin{equation*}
	|X| E_1 E_0^* E_1^* = E_1 + E_1 A^* E_1 + E_1 A_2^* E_1.
\end{equation*}
From this equation, $E_1 A_2^* E_1$ can be expressed in terms of $E_1$, $E_1 E_0^* E_1$ and $E_1 A^* E_1$.
The result follows.
\end{proof}

\begin{lemma}\label{lem:W iff mu}
Let $W$ and $W'$ denote irreducible $T$-modules with endpoint $1$.
Let $\mu$ (resp. $\mu'$) denote the eigenvalue of $E_1^*AE_1^*$ on $E_1^*W$ (resp. $E_1^*W'$).
Then the $T$-modules $W$ and $W'$ are isomorphic if and only if $\mu = \mu'$.
\end{lemma}
\begin{proof}
Pick $0 \neq w \in E_1^*W$ and $0 \neq w' \in E_1^*W'$.
By Lemma \ref{lem:dimE1W=1}(i), we have $E_1^*W = \operatorname{Span}\{w\}$ and $E_1^*W' = \operatorname{Span}\{w'\}$. 
Observe that $w$ (resp. $w'$) is an eigenvector of $E_1^*AE_1^*$ corresponding to the eigenvalue $\mu$ (resp. $\mu'$).

\smallskip
First, suppose that the $T$-modules $W$ and $W'$ are isomorphic.
There exists a $T$-module isomorphism $\sigma: W \to W'$.
Since $w$ (resp. $w'$) is a basis for $E_1^*W$ (resp. $E_1^*W'$), we have $\sigma(w) = \beta w'$ for some $0 \neq \beta \in \mathbb{F}$.
Then  
\begin{equation}
	0 = (\sigma E_1^*AE_1^* - E_1^*AE_1^* \sigma)w = \sigma E_1^*AE_1^* w - E_1^*AE_1^* \sigma(w) = \beta(\mu - \mu')w'.
\end{equation}
Since $\beta \neq 0$ and $w' \neq 0$, it follows that $\mu=\mu'$.

\smallskip
Conversely, suppose that $\mu = \mu'$.
We will show that the $T$-modules $W$ and $W'$ are isomorphic.
Note that $W=Tw$ and $W'=Tw'$ by the irreducibility.
We claim that for any $B\in T$, we have $Bw=0$ if and only if $Bw'=0$.
Observe that $w=E_1^*w$, and thus 
\begin{equation}\label{eq(1):Bw=0 iff...}
	Bw = 0 
	\quad \Longleftrightarrow \quad
	BE_1^*w = 0 
	\quad \Longleftrightarrow \quad
	\lVert BE_1^*w \rVert^2 = 0
	\quad \Longleftrightarrow \quad
	\overline{w}^\top E_1^*\overline{B}^\top B E_1^* w =0.
\end{equation}
Consider the matrix $E_1^*\overline{B}^\top B E_1^*$.
Since $\overline{B}^\top B \in T$, it follows that $E_1^*\overline{B}^\top B E_1^* \in E_1^* T E_1^*$.
By Lemma \ref{NT:Thm2.6}(i) with $r=1$ and Lemma \ref{lem:E1*ME1*,E1TE1}(i), the algebra $E_1^* T E_1^*$ is generated by $E_1^* \mathbb{J} E_1^*$ and $E_1^* A E_1^*$ on $W$, where we note that $E_1^*$ is the multiplicative identity.
Therefore, the restriction of $E_1^*\overline{B}^\top B E_1^*$ to $W$ can be expressed in terms of the restrictions of $E_1^* \mathbb{J} E_1^*$ and $E_1^* A E_1^*$ to $W$.
We now apply this restriction to $w$.
Since $E_1^* \mathbb{J} E_1^*w = 0$ and $E_1^* A E_1^*w = \mu w$, it follows that 
\begin{equation}\label{pf:eq E1*BtBE1*}
	E_1^*\overline{B}^\top B E_1^*w = p(\mu) w,
\end{equation}
for some polynomial $p \in \mathbb{F}[\lambda]$.
Taking the inner product of both sides of \eqref{pf:eq E1*BtBE1*} with $w$ yields
$$
	\overline{w}^\top E_1^*\overline{B}^\top B E_1^*w =  p(\mu) \lVert w \rVert^2.
$$
Therefore, 
\begin{equation}\label{eq(2):Bw=0 iff...}
	\overline{w}^\top E_1^*\overline{B}^\top B E_1^* w =0 
	\quad \Longleftrightarrow \quad	
	p(\mu) \lVert w \rVert^2 = 0
	\quad \Longleftrightarrow \quad	
	p(\mu) = 0.
\end{equation}
By \eqref{eq(1):Bw=0 iff...} and \eqref{eq(2):Bw=0 iff...}, it follows that $Bw=0$ if and only if $p(\mu)=0$.
Similarly, it follows that $Bw'=0$ if and only if $p(\mu)=0$.
Therefore, $Bw=0$ if and only if $Bw'=0$ as claimed.
Next, define the map $\sigma: W \to W'$ that sends $Bw$ to $Bw'$ for all $B \in T$.
By the claim, this map is well-defined.
One verifies that $\sigma$ is a $T$-module homomorphism.
Moreover, $\sigma$ is injective by the claim and surjective by irreducibility.
Consequently, $\sigma$ is a bijection, and hence a $T$-module isomorphism.
The result follows.
\end{proof}

Let $\mathbf{W}$ denote the span of all irreducible $T$-modules with endpoint $1$.

\begin{corollary}\label{cor:W iff mu}
For the $T$-module $\mathbf{W}$, the number of mutually non-isomorphic irreducible $T$-modules with endpoint $1$ is equal to the number of distinct eigenvalues of $E_1^*AE_1^*$ on $E_1^*\mathbf{W}$.
\end{corollary}
\begin{proof}
This follows from Lemma \ref{lem:W iff mu}.
\end{proof}

\begin{lemma}\label{lem:E1*TE1*=<E1*ME1*>}
The algebra $E_1^* T E_1^*$ is generated by $E_1^* M E_1^*$.
\end{lemma}

\begin{proof}
Let $\ell$ denote the number of mutually non-isomorphic irreducible $T$-modules with endpoint $1$.
Let $W_1, W_2, \ldots, W_\ell$ denote mutually non-isomorphic irreducible $T$-modules with endpoint $1$.
For $1 \leq i \leq \ell$, let $\mult_i$ denote the multiplicity of $W_i$ in $V$.
Then we have a direct sum
\begin{equation}\label{eq:ds V:W_i,j}
	\mathbf{W} = \sum_{i=1}^\ell \sum_{j=1}^{\mult_i} W_i^{(j)},
\end{equation}
where each $W_i^{(j)}$ is an irreducible $T$-module isomorphic to $W_i$.
Let $W_0$ denote the primary $T$-module.
Note that the multiplicity of $W_0$ in $V$ is $1$.
Consider the direct sum $W_0 + \mathbf{W}$. 
Applying $E_1^*$ to this sum, we get
\begin{align}
	E_1^*V 
	& = E_1^*W_0 + E_1^*\mathbf{W} \label{eq:ds E1*V(1)} \\
	& = E_1^*W_0 + \sum_{i=1}^\ell \sum_{j=1}^{\mult_i} E_1^*W_i^{(j)} \qquad \text{(direct sum)}. \label{eq:ds E1*V(2)}
\end{align}
By construction, $E_1^* T E_1^*$ acts faithfully on $E_1^*V$.
Moreover, each summand in \eqref{eq:ds E1*V(2)} is an $E_1^* T E_1^*$-submodule and has dimension one by \eqref{eq:Ei*W=1} and Lemma \ref{lem:dimE1W=1}(i).
By Wedderburn's theorem \cite{CR1962} we have an $\mathbb{F}$-algebra isomorphism
\begin{equation}\label{eq:E1*TE1* iso F...F}
	E_1^*TE_1^* \cong \mathbb{F} \oplus \mathbb{F} \oplus \cdots \oplus \mathbb{F} \qquad (\text{$\ell+1$ copies}).
\end{equation}
In particular, $\dim(E_1^* T E_1^*) = \ell+1$.

\smallskip
Next, pick $0 \neq w_0 \in E_1^*W_0$, and for $1 \leq i \leq \ell$ and $1 \leq j \leq \mult_i$, pick $0 \neq w_i^{(j)} \in E_1^* W_i^{(j)}$ 
Observe that 
\begin{equation}\label{eq: basis E1*V wi,j}
	\{ w_0 \} \cup \{ w_i^{(j)} \mid 1 \leq i \leq \ell, \quad 1 \leq j \leq \mult_i \}
\end{equation}
forms a basis for $E_1^*V$.
Define the matrix $K \in \Mat_X(\mathbb{F})$ by
\begin{equation}\label{def:mat K}
	K = E_1^*\mathbb{J}E_1^* + E_1^*AE_1^*.
\end{equation}
Observe that $K  \in E_1^* M E_1^*$.
We describe the action of $K$ on $E_1^*V$.
Recall the direct sum \eqref{eq:ds E1*V(1)} of $E_1^*V$.
First, consider the action of $K$ on $E_1^*W_0$. 
Note that $w_0$ is a scalar multiple of $E_1^*\mathds{1}$.
By construction, we have $K w_0 = (k + a_1)w_0$, where $k$ is the valency of $\Gamma$ and $a_1$ is the intersection number of $\Gamma$.
We denote $\mu_0 := k + a_1$.
Next, consider the action of $K$ on $E_1^* \mathbf{W}$.
Note that $E_1^* \mathbf{W}$ is the orthogonal complement of $E_1^*W_0$ in $E_1^*V$.
Since $E_1^*\mathbb{J}E_1^*$ acts as zero on $E_1^* \mathbf{W}$, the action of $K$ on $E_1^* \mathbf{W}$ coincides with the action of $E_1^*AE_1^*$ on $E_1^* \mathbf{W}$.
By Corollary \ref{cor:W iff mu}, the matrix $E_1^*AE_1^*$ has $\ell$ mutually distinct eigenvalues $\{\mu_i\}_{i=1}^\ell$ on $E_1^* \mathbf{W}$.
For $1 \leq i \leq \ell$, each scalar $\mu_i$ is an eigenvalue of the local graph $\Gamma(x)$ and thus $\mu_i < \mu_0$.
By these comments, the matrix $K$ has $\ell+1$ mutually distinct eigenvalues $\{\mu_i\}_{i=0}^\ell$ on $E_1^*V$.
Since $K$ is real and symmetric, its action on $E_1^*V$ has minimal polynomial of degree $\ell+1$.
Hence, the matrices
\begin{equation}\label{eq:mat K^i}
	E_1^*, \ K, \ K^2, \ \ldots \ , \  K^\ell
\end{equation}
are linearly independent.  
By this and since $\dim(E_1^* T E_1^*)=\ell + 1$, the matrices \eqref{eq:mat K^i} form a basis for  $E_1^*TE_1^*$.
Therefore, $K$ generates $E_1^* T E_1^*$.
By this and since $K \in E_1^* M E_1^*$, the result follows.
\end{proof}

\begin{proposition}\label{lem:E*1TE*1}
Every element of $E_1^*TE_1^*$ is symmetric.
\end{proposition}
\begin{proof}
By Lemma~\ref{lem: E*AE*sym}(i), every element of $E_1^* M E_1^*$ is symmetric.
By Proposition~\ref{lem:E1*TE1* comm} and Lemma~\ref{lem:E1*TE1*=<E1*ME1*>}, the algebra $E_1^* T E_1^*$ is commutative and generated by $E_1^* M E_1^*$.
The result follows.
\end{proof}

We have described the algebra $E_1^*T E_1^*$. 
We now discuss its structure in more detail.

\begin{lemma}[Reduction rules, {\cite[Lemmas 6.3, 6.5]{Ter2024}}]\label{lem:reduction} 
The following {\rm(i)}--{\rm(iii)} hold.
\begin{itemize}
\setlength{\itemsep}{0pt}
	\item[\rm(i)] $E_0 E_1^* E_0 = |X|^{-1} k E_0$, where $k$ is the valency of $\Gamma$.
	\item[\rm(ii)] $E_0 E_1^* A = \sum_{h=0}^D p^h_{1,1} E_0E_h^*$, where $p^h_{ij}$ are the intersection numbers of $\Gamma$.
	\item[\rm(iii)] $A E_1^* E_0 = \sum_{h=0}^D p^h_{1,1} E_h^* E_0$.
\end{itemize}
\end{lemma}
\begin{proof}
Set $i=1$ and $j=1$ in \cite[Lemma 6.3(ii),(iv), Lemma 6.5(iv)]{Ter2024}.
\end{proof}

\begin{lemma}\label{lem:E*1TE*1gen}
The following {\rm(i)}--{\rm(iv)} hold.
\begin{itemize}
\setlength{\itemsep}{0pt}
	\item[\rm(i)] The algebra $E_1^*TE_1^*$ is generated by $E_1^* \mathbb{J} E_1^*$, $E_1^*AE_1^*$.
	\item[\rm(ii)] $(E_1^* \mathbb{J} E_1^*)^2 = k E_1^* \mathbb{J} E_1^*$.
	\item[\rm(iii)] $(E_1^* \mathbb{J} E_1^*)(E_1^* A E_1^*) = a_1E_1^* \mathbb{J} E_1^* = (E_1^* A E_1^*)(E_1^* \mathbb{J} E_1^*)$, where $a_1$ is from \eqref{eq:int num abc}.
	\item[\rm(iv)] $E_1^* \mathbb{J} E_1^*$ is a basis for a (two-sided) ideal of $E_1^* T E_1^*$.
\end{itemize}
\end{lemma}
\begin{proof}
(i): This follows from Lemmas \ref{lem:E1*ME1*,E1TE1}(i) and \ref{lem:E1*TE1*=<E1*ME1*>}.

\noindent
(ii): Recall $\mathbb{J} = |X|E_0$. 
Using this and Lemma \ref{lem:reduction}(i), we have
$$
	(E_1^* \mathbb{J} E_1^*)^2
	= |X|^2 E_1^* (E_0 E_1^* E_0) E_1^* 
	= k |X| E_1^* E_0 E_1^*
	= k E_1^* \mathbb{J} E_1^*.
$$

\noindent
(iii): We have
\begin{align*}
	(E_1^* \mathbb{J} E_1^*)(E_1^* A E_1^*) 
	& = |X| E_1^* (E_0 E_1^* A) E_1^* \\
	& = |X| E_1^* \left( \sum_{h=0}^D p^h_{1,1} E_0E_h^* \right) E_1^*  && \text{(by Lemma \ref{lem:reduction}(ii))}\\
	& = |X| E_1^* \left( p^0_{1,1} E_0E_0^* + p^1_{1,1} E_0E_1^* + p^2_{1,1} E_0E_2^* \right) E_1^* \\
	& = a_1 E_1^* \mathbb{J} E_1^* && \text{($a_1=p^1_{1,1}$ from \eqref{eq:int num abc})}.
\end{align*}
Similarly, using Lemma \ref{lem:reduction}(iii) we obtain $(E_1^* A E_1^*)(E_1^* \mathbb{J} E_1^*) = a_1 E_1^* \mathbb{J} E_1^*$. 

\noindent
(iv): This is immediate from (i)--(iii).
\end{proof}

We now discuss the algebra $E_1TE_1$.

\begin{lemma}\label{lem:E1TE1=<E1M*E1>}
The algebra $E_1 T E_1$ is generated by $E_1 M^* E_1$.
\end{lemma}
\begin{proof}
This is similar to Lemma~\ref{lem:E1*TE1*=<E1*ME1*>}.
\end{proof}

\begin{proposition}\label{lem:E1TE1}
Every element of $E_1TE_1$ is symmetric.
\end{proposition}
\begin{proof}
This is similar to Proposition~\ref{lem:E*1TE*1}.
\end{proof}

\begin{lemma}[Reduction rules, {\cite[Lemmas 6.4, 6.6]{Ter2024}}]\label{lem:reduction(dual)} 
The following {\rm(i)}--{\rm(iii)} hold.
\begin{itemize}
\setlength{\itemsep}{0pt}
	\item[\rm(i)] $E_0^* E_1 E_0^* = |X|^{-1} m_1 E_0^*$, where $m_1$ is the trace of $E_1$.
	\item[\rm(ii)] $E_0^* E_1 A^* = \sum_{h=0}^D q^h_{1,1} E_0^* E_h$, where $q^h_{i,j}$ are the Krein parameters of $\Gamma$.
	\item[\rm(iii)] $A^* E_1 E_0^* = \sum_{h=0}^D q^h_{1,1} E_h E_0^*$.
\end{itemize}
\end{lemma}
\begin{proof}
Set $i=1$ and $j=1$ in \cite[Lemma 6.4(ii),(iv), Lemma 6.6(iv)]{Ter2024}.
\end{proof}

\begin{lemma}\label{lem:E1TE1gen}
The following {\rm(i)}--{\rm(iv)} hold.
\begin{itemize}
\setlength{\itemsep}{0pt}
	\item[\rm(i)] The algebra $E_1TE_1$ is generated by $E_1 E_0^* E_1$, $E_1A^*E_1$.
	\item[\rm(ii)] $(E_1 E_0^* E_1)^2 = m_1 |X|^{-1} E_1 E_0^* E_1$, where $m_1$ is the trace of $E_1$.
	\item[\rm(iii)] $(E_1 E_0^* E_1)(E_1 A^* E_1) = a_1^*E_1 E_0^* E_1 = (E_1 A^* E_1)(E_1 E_0^* E_1)$, where $a_1^*$ is from \eqref{eq:dual int num abc}.
	\item[\rm(iv)] $E_1 E_0^* E_1$ is a basis for a (two-sided) ideal of $E_1 T E_1$.
\end{itemize}

\end{lemma}
\begin{proof}
(i): This follows from Lemmas \ref{lem:E1*ME1*,E1TE1}(ii) and \ref{lem:E1TE1=<E1M*E1>}.

\noindent
(ii): By Lemma \ref{lem:reduction(dual)}(i), we have
$$
	(E_1 E_0^* E_1)^2
	= E_1 (E_0^* E_1 E_0^*) E_1
	= m_1 |X|^{-1} E_1 E_0^* E_1.
$$

\noindent
(iii): We have
\begin{align*}
	(E_1 E_0^* E_1)(E_1 A^* E_1)
	& = E_1 (E_0^* E_1 A^*) E_1 \\
	& = E_1 \left( \sum_{h=1}^D q^h_{1,1} E_0^*E_h\right) E_1 && \text{(By Lemma \ref{lem:reduction(dual)}(ii))} \\
	& = E_1 (q^0_{1,1}E_0^*E_0 + q^1_{1,1}E_0^*E_1 + q^2_{1,1}E_0^*E_2)E_1 \\
	& = a_1^* E_1 E_0^*E_1 && \text{($a_1^* = q^1_{1,1}$ from \eqref{eq:dual int num abc})}.
\end{align*}
Similarly, using Lemma \ref{lem:reduction(dual)}(iii) we obtain $(E_1 A^* E_1)(E_1 E_0^* E_1) = a_1^* E_1 E_0^* E_1$.

\noindent
(iv): This is immediate from (i)--(iii). 

\end{proof}

\section{The algebras $E_D^*TE_D^*$ and $E_DTE_D$}\label{sec:E*DTE*D}
We continue to discuss the $Q$-polynomial distance-regular graph $\Gamma$.
Fix $x\in X$ and write $T=T(x)$.
In this section, we discuss the algebras $E_D^*TE_D^*$ and $E_DTE_D$.

\begin{lemma}\label{lem:dimEDW=1}
Let $W$ denote an irreducible $T$-module with endpoint $r$, dual endpoint $s$, and diameter $d$.
Note that $r+d \leq D$ and $s+d \leq D$.
The following {\rm(i)}--{\rm(iv)} hold.
\begin{itemize}
\setlength\itemsep{0pt}
	\item[{\rm (i)}] If $r+d = D$, then $E_D^*W$ has dimension one.
	\item[{\rm (ii)}] If $r+d < D$, then $E_D^*W=0$.
	\item[{\rm (iii)}] If $s+d = D$, then $E_D W$ has dimension one.
	\item[{\rm (iv)}] If $s+d < D$, then $E_D W=0$.	
\end{itemize}

\end{lemma}
\begin{proof}
(i),(iii): These follow from Theorem \ref{mainthm} and Lemma \ref{lem: dim sym}.\\
(ii),(iv): These follow by construction. 
\end{proof}

The following lemmas are analogues to Lemmas~\ref{NT:Thm2.6} and~\ref{NT:Thm2.6 dual}.

\begin{lemma}[{\cite[Theorem 2.6]{NT2008LAA}}]\label{NT:Thm2.6 sym version}
Let $W$ denote an irreducible $T$-module with endpoint $r$, dual endpoint $s$, and diameter $d$.
Then the following {\rm(i)}--{\rm(iii)} hold on $W$.
\begin{itemize}
\setlength{\itemsep}{0pt}
	\item[\rm(i)] The algebra $E_{r+d}^* T E_{r+d}^*$ is generated by $E_{r+d}^* M E_{r+d}^*$.
	\item[\rm(ii)] The elements of $E_{r+d}^* M E_{r+d}^*$ mutually commute.
	\item[\rm(iii)] The algebra $E_{r+d}^* T E_{r+d}^*$ is commutative.
\end{itemize}
\end{lemma}

\begin{proof} 
This follows from \cite[Theorem 2.6]{NT2008LAA} with $\{E_{r+i}^*\}_{i=0}^d$ replaced by $\{E_{r+d-i}^*\}_{i=0}^d$.
\end{proof}

\begin{lemma}[{\cite[Theorem 2.6]{NT2008LAA}}]
Let $W$ denote an irreducible $T$-module with endpoint $r$, dual endpoint $s$, and diameter $d$.
Then the following {\rm(i)}--{\rm(iii)} hold on $W$.
\begin{itemize}
\setlength{\itemsep}{0pt}
	\item[\rm(i)] The algebra $E_{s+d} T E_{s+d}$ is generated by $E_{s+d} M^* E_{s+d}$.
	\item[\rm(ii)] The elements of $E_{s+d} M^* E_{s+d}$ mutually commute.
	\item[\rm(iii)] The algebra $E_{s+d} T E_{s+d}$ is commutative.
\end{itemize}
\end{lemma}
\begin{proof} 
This is similar to Lemma \ref{NT:Thm2.6 sym version}.
\end{proof}

\begin{proposition}\label{lem:ED*TED* comm}
The algebras $E_D^* T E_D^*$ and $E_D T E_D$ are commutative.
\end{proposition}
\begin{proof}
We first show that $E_D^* T E_D^*$ is commutative.  
Since $V$ is an orthogonal direct sum of irreducible $T$-modules, it suffices to show that the action of $E_D^* T E_D^*$ on any irreducible $T$-module is commutative.
Let $W$ be an irreducible $T$-module.
By Lemma~\ref{lem:dimEDW=1}(i),(ii), either $\dim(E_D^* W) = 1$ or $E_D^*W=0$.
In either case, the action of $E_D^* T E_D^*$ on $W$ is commutative.
We have shown that $E_D^* T E_D^*$ is commutative.
The argument for $E_D T E_D$ is similar.
\end{proof}

Let $W$ denote an irreducible $T$-module with endpoint $r$ and diameter $d$ satisfying $r+d=D$.
Pick $0 \neq w \in E_D^*W$.
Since $E_D^*W$ is invariant under $E_D^* M E_D^*$ and by Lemma \ref{lem:dimEDW=1}(i), there exist scalars $\{ \varphi_h \}_{h=1}^D$ in $\mathbb{R}$ such that
\begin{equation}\label{eig seq E*DW}
	E_D^* A_h E_D^*w =  \varphi_h w \qquad (1 \leq h \leq D).
\end{equation}
We call $\{ \varphi_h \}_{h=1}^D$ the \emph{local eigenvalue sequence} of $E_D^*W$.

\begin{lemma}\label{lem:W iff mu, E*D version}
Let $W$ and $W'$ denote irreducible $T$-modules with endpoint $r$ and diameter $d$ satisfying $r+d=D$.
Then the $T$-modules $W$ and $W'$ are isomorphic if and only if the local eigenvalue sequences of $E_D^*W$ and $E_D^*W'$ coincide.
\end{lemma}
\begin{proof}
Let $\{ \varphi_h \}_{h=1}^D$ (resp. $\{ \varphi'_h \}_{h=1}^D$) denote the local eigenvalue sequence of $E_D^*W$ (resp. $E_D^*W'$).
Pick $0 \neq w \in E_D^*W$ and $0 \neq w' \in E_D^*W'$.
By Lemma \ref{lem:dimEDW=1}(i), we have $E_D^*W = \operatorname{Span}\{w\}$ and $E_D^*W' = \operatorname{Span}\{w'\}$. 
Observe that $w$ (resp. $w'$) is an eigenvector of $E_D^*A_hE_D^*$ corresponding to the eigenvalue $\varphi_h$ (resp. $\varphi'_h$) for $1 \leq h \leq D$.

\smallskip
First, suppose that the $T$-modules $W$ and $W'$ are isomorphic.
There exists a $T$-module isomorphism $\sigma: W \to W'$.
Since $w$ (resp. $w'$) is a basis for $E_D^*W$ (resp. $E_D^*W'$), we have $\sigma(w) = \beta w'$ for some $0 \neq \beta \in \mathbb{F}$.
Then for $1 \leq h \leq D$
\begin{equation*}
	0 = (\sigma E_D^*A_hE_D^* - E_D^*A_hE_D^* \sigma)w = \sigma (E_D^*A_hE_D^* w) - E_D^*A_hE_D^* \sigma(w) = \beta(\varphi_h - \varphi'_h)w'.
\end{equation*}
Since $\beta \neq 0$ and $w' \neq 0$, it follows that $\varphi_h=\varphi'_h$ for $1 \leq h \leq D$.

\smallskip
Conversely, suppose that $\varphi_h = \varphi'_h$ for $1 \leq h \leq D$.
We will show that $W$ and $W'$ are isomorphic as $T$-modules.
Observe that $W=Tw$ and $W'=Tw'$ by the irreducibility.
We first claim that for any $B\in T$, we have $Bw=0$ if and only if $Bw'=0$.
Observe that $w=E_D^*w$, and thus 
\begin{equation}\label{eq(1):Bw=0 iff...; D version}
	Bw = 0 
	\ \ \Longleftrightarrow \ \
	BE_D^*w = 0 
	\ \ \Longleftrightarrow \ \
	\lVert BE_D^*w \rVert^2 = 0
	\ \ \Longleftrightarrow \ \
	\overline{w}^\top E_D^*\overline{B}^\top B E_D^* w =0.
\end{equation}
Consider the matrix $E_D^*\overline{B}^\top B E_D^*$.
Since $\overline{B}^\top B \in T$, it follows that $E_D^*\overline{B}^\top B E_D^* \in E_D^*TE_D^*$. 
By Lemma~\ref{NT:Thm2.6 sym version}(i), and since $M$ is spanned by $\{A_h\}_{h=0}^D$ with $A_0 = I$, the algebra $E_D^* T E_D^*$ is generated  on $W$ by the elements $\{E_D^* A_h E_D^*\}_{h=1}^D$.
Therefore, the restriction of $E_D^*\overline{B}^\top B E_D^*$ to $W$ can be expressed in terms of the restrictions of $\{E_D^* A_h E_D^*\}_{h=1}^D$ to $W$.
Hence, there exists a polynomial $\phi \in \mathbb{F}[\lambda_1, \lambda_2, \ldots, \lambda_D]$ such that on $W$,
\begin{equation}\label{pf:eq E*DBtBD*}
	E_D^*\overline{B}^\top B E_D^* =  \phi (E_D^* A E_D^*, \ E_D^* A_2 E_D^*, \ \ldots \ , \ E_D^* A_D E_D^*).
\end{equation}
We now apply $E_D^*\overline{B}^\top B E_D^*$ to $w$. 
Using \eqref{eig seq E*DW} and \eqref{pf:eq E*DBtBD*}, we find that on $W$,
\begin{equation}\label{pf:eq E*DBtBD*w}
	E_D^*\overline{B}^\top B E_D^*w 
	= \phi (E_D^* A E_D^*, \ E_D^* A_2 E_D^*, \ \ldots \ , \ E_D^* A_D E_D^*)w 
	= \phi (\varphi_1, \varphi_2, \ldots, \varphi_D) w.
\end{equation}
Take the inner product of each term in \eqref{pf:eq E*DBtBD*w} with $w$ to obtain
$$
	\overline{w}^\top E_D^*\overline{B}^\top B E_D^* w 
	= \phi (\varphi_1, \varphi_2, \ldots, \varphi_D) \lVert w \rVert^2.
$$
Therefore, 
\begin{equation}\label{eq(2):Bw=0 iff...; D version}
\begin{split}
	\overline{w}^\top E_D^*\overline{B}^\top B E_D^* w =0 
	& \ \ \Longleftrightarrow \ \	
	\phi (\varphi_1, \varphi_2, \ldots, \varphi_D) \lVert w \rVert^2 = 0 \\
	& \ \ \Longleftrightarrow \ \
	\phi (\varphi_1, \varphi_2, \ldots, \varphi_D) = 0.
\end{split}
\end{equation}
By \eqref{eq(1):Bw=0 iff...; D version} and \eqref{eq(2):Bw=0 iff...; D version}, we obtain $Bw=0$ if and only if $\phi (\varphi_1, \varphi_2, \ldots, \varphi_D)=0$.
Similarly, we obtain $Bw'=0$ if and only if $\phi (\varphi_1, \varphi_2, \ldots, \varphi_D)=0$.
Therefore, $Bw=0$ if and only if $Bw'=0$ as claimed.
Next, define the map $\sigma: W \to W'$ that sends $Bw$ to $Bw'$ for all $B \in T$.
By the claim, this map is well-defined.
One verifies that $\sigma$ is a $T$-module homomorphism.
Moreover, $\sigma$ is injective by the claim and surjective by irreducibility.
Consequently, $\sigma$ is a bijection, and hence a $T$-module isomorphism.
The result follows.
\end{proof}

\begin{lemma}\label{lem:ED*TED*=<ED*MED*>}
The algebra $E_D^* T E_D^*$ is generated by $E_D^* M E_D^*$.
\end{lemma}

\begin{proof}
Let $\ell$ denote the number of mutually non-isomorphic irreducible $T$-modules with endpoint $r$ and diameter $d$ satisfying $r+d = D$.
Note that $\ell \geq 1$ since the primary module has endpoint $0$ and diameter $D$.
Let $W_1, \ldots, W_\ell$ denote mutually non-isomorphic irreducible $T$-modules with endpoint $r$ and diameter $d$ satisfying $r+d = D$.
For $1 \leq i \leq \ell$, let $\mult_i$ denote the multiplicity of $W_i$ in $V$. 
By construction, $V$ contains a direct sum
\begin{equation}\label{eq:ds W_ij}
	\sum_{i=1}^\ell \sum_{j=1}^{\mult_i} W_i^{(j)}
\end{equation}
where each $W_i^{(j)}$ is an irreducible $T$-module isomorphic to $W_i$.
Applying $E_D^*$ to \eqref{eq:ds W_ij}, we get
\begin{equation}\label{eq:ds ED*V}
	E_D^*V = \sum_{i=1}^\ell \sum_{j=1}^{\mult_i} E_D^*W_i^{(j)} \qquad \text{(direct sum)}.
\end{equation}
By construction, $E_D^* T E_D^*$ acts faithfully on $E_D^*V$.
Moreover, each summand in \eqref{eq:ds ED*V} is an $E_D^* T E_D^*$-submodule and has dimension one by Lemma \ref{lem:dimEDW=1}(i).
By Wedderburn's theorem \cite{CR1962} we have an $\mathbb{F}$-algebra isomorphism
\begin{equation}\label{eq:ED*TED* iso F...F}
	E_D^* T E_D^* \cong \mathbb{F} \oplus \mathbb{F} \oplus \cdots \oplus \mathbb{F} \qquad \text{($\ell$ copies)}.
\end{equation}
In particular, $\dim(E_D^* T E_D^*) = \ell$.

\smallskip
Next, for $1 \leq i \leq \ell$ and $1 \leq j \leq \mult_i$, pick $0 \neq w_i^{(j)} \in E_D^* W_i^{(j)}$.
Observe that 
\begin{equation}\label{eq: basis ED*V wi,j}
	\{ w_i^{(j)} \mid 1 \leq i \leq \ell, \quad 1 \leq j \leq \mult_i \}
\end{equation}
forms a basis for $E_D^*V$.
For each $1 \leq i \leq \ell$, let $\{ \varphi_{i,h} \}_{h=1}^D$ denote the local eigenvalue sequence of $E_D^*W_i$; that is, the scalars $\{ \varphi_{i,h} \}_{h=1}^D$ satisfy
\begin{equation}\label{eq:ED*AhED*.w_ij}
	E_D^* A_h E_D^* w_i^{(j)} = \varphi_{i,h} w_i^{(j)} \qquad (1 \leq h \leq D, \ 1 \leq j \leq \mult_i).
\end{equation}
For $1 \leq i \leq \ell$, we view the local eigenvalue sequence $\{\varphi_{i,h}\}_{h=1}^D$ of $E_D^*W_i$ as a vector
\begin{equation}
	u_i := (\varphi_{i,1}, \; \varphi_{i,2}, \; \ldots, \; \varphi_{i,D}) \in \mathbb{F}^D.
\end{equation}
Since the $T$-modules $W_1, W_2, \ldots , W_\ell$ are mutually non-isomorphic, it follows from Lemma \ref{lem:W iff mu, E*D version} that the vectors $\{ u_i \}_{i=1}^\ell$ are mutually distinct.
By this fact and since the field $\mathbb{F}$ is infinite, there exists a linear transformation $f: \mathbb{F}^D \to \mathbb{F}$ such that the scalars $\{ f(u_i) \}_{i=1}^\ell$ are mutually distinct. 
Write this linear transformation as
$$
	f( \lambda_1, \lambda_2, \ldots, \lambda_D ) = \sum_{i=1}^D \alpha_i \lambda_i
$$
for some $\alpha_1, \ldots, \alpha_D \in \mathbb{F}$.
Define the matrix $\Delta \in \Mat_X(\mathbb{F})$ by
\begin{equation}\label{eq:Delta}
	\Delta = \sum_{h=1}^D \alpha_h E_D^* A_h E_D^*.
\end{equation}
Observe that $\Delta \in E_D^* M E_D^*$.
Moreover, by \eqref{eq:ED*AhED*.w_ij} 
\begin{equation}
	\Delta w_i^{(j)} = f(u_i) w_i^{(j)} \qquad (1 \leq i \leq \ell, \quad 1 \leq j \leq \mult_i).
\end{equation}
This implies that the action of $\Delta$ on $E_D^*V$ has $\ell$ distinct eigenvalues $\{f(u_i)\}_{i=1}^\ell$.
Since $\Delta$ is real and symmetric, its action on $E_D^*V$ has minimal polynomial of degree $\ell$.
It follows that the matrices
\begin{equation}\label{eq:mat Delta^i}
	E_D^*, \ \Delta, \ \Delta^2, \ldots, \Delta^{\ell-1}
\end{equation}
are linearly independent.
By this and since $\dim(E_D^* T E_D^*)=\ell$, the matrices \eqref{eq:mat Delta^i} form a basis for $E_D^* T E_D^*$.
Therefore, $\Delta$ generates $E_D^* T E_D^*$.
By this and since $\Delta \in E_D^* M E_D^*$, the result follows.
\end{proof}

\begin{proposition}\label{lem:E*DTE*D}
Every element of $E_D^*TE_D^*$ is symmetric.
\end{proposition}
\begin{proof}
By Lemma~\ref{lem: E*AE*sym}(i), every element of $E_D^* M E_D^*$ is symmetric.
By Proposition~\ref{lem:ED*TED* comm} and Lemma~\ref{lem:ED*TED*=<ED*MED*>}, the algebra $E_D^* T E_D^*$ is commutative and generated by $E_D^* M E_D^*$.
The result follows.
\end{proof}

We now discuss the algebra $E_DTE_D$.

\begin{lemma}
The algebra $E_D T E_D$ is generated by $E_D M^* E_D$.
\end{lemma}
\begin{proof}
This is similar to Lemma \ref{lem:ED*TED*=<ED*MED*>}.
\end{proof}

\begin{proposition}\label{lem:EDTED sym}
Every element of $E_D T E_D$ is symmetric.
\end{proposition}
\begin{proof}
This is similar to Proposition \ref{lem:E*DTE*D}.
\end{proof}

We now prove Theorem \ref{thm:four algebras}

\begin{proof}[Proof of Theorem \ref{thm:four algebras}]
The result follows from Propositions \ref{lem:E1*TE1* comm}, \ref{lem:E*1TE*1}, \ref{lem:ED*TED* comm}, and \ref{lem:EDTED sym}.
\end{proof}

Recall the definition of association schemes; see \cite[Section 2.2]{BI1984}.
We conclude this paper with two problems.

\begin{problem}\label{prob1}
Find necessary and sufficient conditions under which $E_1^*TE_1^*$ is isomorphic to the Bose-Mesner algebra of a commutative association scheme with vertex set $\Gamma_1(x)$.
\end{problem}

\begin{problem}\label{prob2}
Find necessary and sufficient conditions under which $E_D^*TE_D^*$ is isomorphic to the Bose-Mesner algebra of a commutative association scheme with vertex set $\Gamma_D(x)$.
\end{problem}

\begin{example}\label{ex:BM Gamma(x)}
We present some graphs $\Gamma$ such that for every vertex $x$, the algebra $E_1^*TE_1^*$ is isomorphic to the Bose-Mesner algebra of a commutative association scheme on $\Gamma_1(x)$.
\begin{itemize}
	\item[(i)] Assume $\Gamma$ is the Hamming graph $H(D,N)$, where $D\geq 2$ and $N\geq 3$; see \cite[Section 9.2]{BCN}.
	The vertex set $X$ of $\Gamma$ consists of $D$-tuples with entries from the set $\{1,2, \ldots, N\}$.
	Two vertices in $X$ are adjacent whenever they differ in exactly one coordinate. 
	For any $y, z \in X$, the distance $\partial(y,z)$ is equal to the number of coordinates at which $y$ and $z$ differ.
	Fix a vertex $x \in X$ and write $T=T(x)$.
	Consider the first subconstituent $\Gamma(x)$.
	The subgraph of $\Gamma$ induced on $\Gamma(x)$ is a disjoint union of $D$ copies of the complete graph $K_{N-1}$.
	This subgraph corresponds to the commutative association scheme $\mathcal{X} = (\Gamma(x), \{R_i\}_{i=0}^2)$, where 
	\begin{align*}
		R_0 & = \{ (y,y) \mid y \in \Gamma(x) \}, \\
		R_1 & = \{ (y,z) \mid y,z \in \Gamma(x), \ \partial(y,z) = 1 \}, \\
		R_2  & = \{ (y,z) \mid y,z \in \Gamma(x), \ y \neq z, \ \partial(y,z) \neq 1 \}.
	\end{align*}
	The Bose-Mesner algebra of $\mathcal{X}$ has dimension $3$ and is isomorphic to $E_1^* T E_1^*$.

	\item[(ii)] 	Assume $\Gamma$ is the Johnson graph $J(N, D)$, where $D\geq 2$ and $N \geq 2D$; see \cite[Section 9.1]{BCN}.
	The vertex set $X$ of $\Gamma$ consists of the $D$-element subsets of the set $\{1,2,.\ldots, N\}$.
	Two vertices $y, z \in X$ are adjacent whenever $|y \cap z| = D - 1$.
	For any $y, z \in X$, the distance $\partial(y, z)$ is equal to $D - |y \cap z|$. 
	Fix a vertex $x \in X$ and write $T=T(x)$.
	Consider the first subconstituent $\Gamma(x)$.
	The subgraph of $\Gamma$ induced on $\Gamma(x)$ is isomorphic to the Cartesian product $K_{D} \times K_{N-D}$.
	This subgraph corresponds to the commutative association scheme $\mathcal{X} = (\Gamma(x), \{R_i\}_{i=0}^3)$, where 
	\begin{align*}
		R_0 & = \{ (y, y) \mid y \in \Gamma(x) \}, \\
		R_1 & = \{ (y, z) \mid y, z \in \Gamma(x), \ x \cap y = x \cap z \}, \\
		R_2 & = \{ (y, z) \mid y, z \in \Gamma(x), \ x \cup y = x \cup z \}, \\
		R_3 & = \{ (y, z) \mid y, z \in \Gamma(x), \ \partial(x,y) = 2 \}.
	\end{align*}
	The Bose-Mesner algebra of $\mathcal{X}$ has dimension $4$ and is isomorphic to $E_1^* T E_1^*$.

	\item[(iii)] 
	Assume $\Gamma$ is the Grassmann graph $J_q(N,D)$, where $D \geq 2$ and $N > 2D$; see \cite[Section 9.3]{BCN}.
	Let $\mathbb{V}$ denote an $N$-dimensional vector space over $\mathbb{F}_q$.
	The vertex set $X$ of $\Gamma$ consists of the $D$-dimensional subspaces of $\mathbb{V}$.
	Two vertices $y, z \in X$ are adjacent whenever $\dim(y \cap z) = D - 1$.
	For any $y, z \in X$, the distance $\partial(y,z)$ is equal to $D - \dim(y \cap z)$.
	Fix a vertex $x \in X$ and write $T = T(x)$.
	Consider the first subconstituent $\Gamma(x)$.
 	The subgraph of $\Gamma$ induced on $\Gamma(x)$ is isomorphic to the $q$-clique extension of the Cartesian product $K_{\mathsf{n}} \times K_{\mathsf{m}}$, where $\mathsf{n} = \left[ \begin{smallmatrix} D \\ 1 \end{smallmatrix} \right]_q$ and 	$\mathsf{m} = \left[ \begin{smallmatrix} N - D \\ 1 \end{smallmatrix} \right]_q$; 	see \cite[Section 4]{GK2025}.
	This subgraph corresponds to the commutative association scheme $\mathcal{X} = (\Gamma(x), \{R_i\}_{i=0}^4)$, where 
	\begin{align*}
		R_0 & = \{ (y, y) \mid y \in \Gamma(x) \}, \\
		R_1 & = \{ (y, z) \mid y, z \in \Gamma(x), \ x \cap y = x \cap z, \ x + y = x + z \}, \\
		R_2 & = \{ (y, z) \mid y, z \in \Gamma(x), \ x \cap y = x \cap z, \ x + y \neq x + z \}, \\
		R_3 & = \{ (y, z) \mid y, z \in \Gamma(x), \ x \cap y \neq x \cap z, \ x + y = x + z \}, \\
		R_4 & = \{ (y, z) \mid y, z \in \Gamma(x), \ \partial(x,y) = 2 \}.
	\end{align*}
	The Bose-Mesner algebra of $\mathcal{X}$ has dimension $5$ and is isomorphic to $E_1^* T E_1^*$.

\end{itemize}
\end{example}

\begin{example} We present some graphs $\Gamma$ such that for every vertex $x$, the algebra $E_D^*TE_D^*$ is isomorphic to the Bose-Mesner algebra of a commutative association scheme on $\Gamma_D(x)$.
\begin{itemize}
	\item[(i)] Assume $\Gamma$ is the Hamming graph $H(D,N)$ as in Example \ref{ex:BM Gamma(x)}(i).
	Fix a vertex $x \in X$ and write $T = T(x)$.
	Consider the last subconstituent $\Gamma_D(x)$.
	The subgraph of $\Gamma$ induced on $\Gamma_D(x)$ is isomorphic to the Hamming graph $H(D, N-1)$.
	This subgraph corresponds to the commutative association scheme $\mathcal{X} = (\Gamma_D(x), \{R_i\}_{i=0}^D)$, where
	\begin{equation*}
	R_i =  \{ (y,z) \mid y, z \in \Gamma_D(x), \ \partial (y,z) = i \} \qquad (0 \leq i \leq D).
	\end{equation*}
	The restrictions of $E_D^* A_i E_D^*$ $(0 \leq i \leq D)$ to $E_D^*V$ form a basis for the Bose-Mesner algebra of $\mathcal{X}$.
	By Lemma \ref{lem:ED*TED*=<ED*MED*>}, this Bose-Mesner algebra is isomorphic to $E_D^* T E_D^*$.

	\item[(ii)] Assume $\Gamma$ is the Johnson graph $\Gamma = J(N,D)$ as in Example \ref{ex:BM Gamma(x)}(ii).
	Fix a vertex $x \in X$ and write $T = T(x)$.
	Consider the last subconstituent $\Gamma_D(x)$.
	The subgraph of $\Gamma$ induced on $\Gamma_D(x)$ is isomorphic to the Johnson graph $J(N - D, D)$.
	The subgraph corresponds to the commutative association scheme $\mathcal{X} = (\Gamma_D(x), \{R_i\}_{i=0}^d)$, where $d = \min(D, N - 2D)$ and
	\begin{equation*}
	R_i =  \{ (y,z) \mid y, z \in \Gamma_D(x), \  |y \cap z|  = d - i \} \qquad (0 \leq i \leq d).
	\end{equation*}
	The restrictions of $E_D^* A_i E_D^*$ $(0 \leq i \leq d)$ to $E_D^*V$ form a basis for the Bose-Mesner algebra of $\mathcal{X}$.
	By Lemma \ref{lem:ED*TED*=<ED*MED*>}, this Bose-Mesner algebra is isomorphic to $E_D^* T E_D^*$.

	\item[(iii)] Assume $\Gamma$ is the Grassmann graph $J_q(2D,D)$ with $D\geq 2$ as in Example \ref{ex:BM Gamma(x)}(iii).  
	Fix a vertex $x \in X$ and write $T = T(x)$.
	Consider the last subconstituent $\Gamma_D(x)$.
	The subgraph of $\Gamma$ induced on $\Gamma_D(x)$ is isomorphic to the bilinear forms graph on the set $D \times D$ matrices over $\mathbb{F}_q$; see \cite[Section 9.5A]{BCN}.
	The subgraph corresponds to the commutative association scheme $\mathcal{X} = (\Gamma_D(x), \{R_i\}_{i=0}^D)$, where 
	\begin{equation*}
	R_i =  \{ (y,z) \mid y, z \in \Gamma_D(x), \  \dim(y \cap z) = D - i \} \qquad (0 \leq i \leq D).
	\end{equation*}	
	The restrictions of $E_D^* A_i E_D^*$ $(0 \leq i \leq D)$ to $E_D^*V$ form a basis for the Bose-Mesner algebra of $\mathcal{X}$.
	By Lemma \ref{lem:ED*TED*=<ED*MED*>}, this Bose-Mesner algebra is isomorphic to $E_D^* T E_D^*$.

\end{itemize}
\end{example}

\begin{remark}
If $E_1^*TE_1^*$ (resp. $E_D^*TE_D^*$) is the Bose-Mesner algebra of a commutative association scheme with vertex set $\Gamma_1(x)$ (resp. $\Gamma_D(x)$), then this association scheme must be symmetric by Proposition \ref{lem:E*1TE*1} (resp. \ref{lem:E*DTE*D}).
\end{remark}

\section*{Declaration of competing interest}
The authors have no relevant financial or non-financial interests to disclose.

\section*{Data availability}
All data generated or analyzed during this study are included in this
published article.

\section*{Acknowledgements}
The authors express their deep gratitude to Paul Terwilliger for his careful reading and for his many valuable comments and suggestions. 
This work was partially completed while B. Fern\'andez and J.-H. Lee were visiting the Department of Mathematics at Kyungpook National University, to which they are grateful for its warm hospitality.
B. Fern\'andez is partially supported by the Slovenian Research Agency through research program P1-0285 and research project Z1-70008.
J. Park is supported by the National Research Foundation of Korea (NRF) grant funded by the Korea government (MSIT) (RS-2024-00356153).


\begin{thebibliography}{10}

	\bibitem{BI1984}
	E.~Bannai, T.~Ito. 
	\newblock Algebraic Combinatorics I: Association Schemes. 
	\newblock Benjamin/Cummings, Menlo Park, CA, 1984.

	\bibitem{BCV2022}
	P.-A.~Bernard, N.~Cramp\'e, L.~Vinet.
	\newblock The Terwilliger algebra of symplectic dual polar graphs, the subspace lattices and $U_q(\mathfrak{sl}_2)$. 
	\newblock Discrete Math. 345 (2022) no. 12, Paper No. 113169, 19pp.

	\bibitem{Biggs}
	N.~Biggs. 
	\newblock Algebraic Graph Theory.
	\newblock Cambridge University Press, London, 1994.

	\bibitem{BCN}
	A.~E.~Brouwer, A.~M.~Cohen, and A.~Neumaier. 
	\newblock Distance-Regular Graphs. 
	\newblock Springer-Verlag, Berlin, 1989.

	\bibitem{Caughman1999}
	J.~S.~Caughman~IV. 
	\newblock The Terwilliger algebras of bipartite $P$- and $Q$-polynomial association schemes. 
	\newblock Discrete Math. 196 (1999) 65--95.	
	
	\bibitem{CR1962}
	C.~Curtis, I.~Reiner. 
	\newblock Representation Theory of Finite Groups and Associative Algebras. 
	\newblock Pure and Applied Mathematics, Vol. XI, Interscience Publishers, a division of John Wiley \& Sons, New York-London, 1962.
	
	\bibitem{DKT2016}
	E.~R.~van~Dam, J.~H.~Koolen, H.~Tanaka. 
	\newblock Distance-regular graphs. 
	\newblock Electron. J. Combin. (2016) DS22.
	
	\bibitem{Delsarte1973}
	P.~Delsarte. 
	\newblock An algebraic approach to the association schemes of coding theory. 
	\newblock Philips Res. Rep. Suppl., No.10 (1973).
	
	\bibitem{Egge2000}
	E.~Egge. 
	\newblock A generalization of the Terwilliger algebra. 
	\newblock J. Algebra 233 (2000) 213--252.
	
	\bibitem{Fernandez2022}
	B.~Fern\'andez.
	\newblock Certain graphs with exactly one irreducible $T$-module with endpoint 1, which is thin.
	\newblock J. Algebraic Combin. 56 (2022), no. 4, 1287--1307.
	
	\bibitem{FH1991}
	W.~Fulton and J.~Harris.
	\newblock Representation Theory: A First Course.
	\newblock Graduate Texts in Mathematics, Vol.~129. Springer, New York, 1991.
	
	\bibitem{GK2015}
	A.~L.~Gavrilyuk, J.~H.~Koolen.
	\newblock The Terwilliger polynomial of a $Q$-polynomial distance-regular graph and its application to pseudo-partition graphs.
	\newblock Linear Algebra Appl. 466 (2015) 117--140.
	
	\bibitem{GK2025}
	A.~L.~Gavrilyuk, J.~H.~Koolen.
	\newblock A characterization of the Grassmann graphs.
	\newblock J. Comb. Theory, Ser. B 171 (2025) 1--27.
	
	\bibitem{Huang2025}
	H.-W.~Huang.
	\newblock An imperceptible connection between the Clebsch-Gordan coefficients of  $U_q(\mathfrak{sl}_2)$ and the Terwilliger algebras of Grassmann graphs.
	\newblock J. Combin. Theory Ser. A 214 (2025) Paper No. 106028, 83pp.
	
	\bibitem{ITT}
	T.~Ito, K.~Tanabe, P.~Terwilliger.
	\newblock {Some algebra related to $P$- and $Q$-polynomial association schemes}.
	\newblock Codes and Association Schemes (Piscataway NJ, 1999), American Mathematical Society, Providence, RI, 2001, pp. 167--192.
	
	\bibitem{INP2011LAA}
	T.~Ito, K.~Nomura, P.~Terwilliger.
	\newblock {A classification of sharp tridiagonal pairs}.
	\newblock Linear Algebra Appl. 435 (2011) 1857--1884.
	
	\bibitem{IT2004JPAA}
	T.~Ito, P.~Terwilliger. 
	\newblock The shape of a tridiagonal pair. 
	\newblock J. Pure Appl. Algebra 188 (2004) 145--160.
	
	\bibitem{IT2007RJ}
	T.~Ito, P.~Terwilliger.
	\newblock Tridiagonal pairs and the quantum affine algebra $U_q(\widehat{\mathfrak{sl}}_2)$.
	\newblock Ramanujan J. 13 (2007) 39--62.
	
	\bibitem{IT2007LAA}
	T.~Ito, P.~Terwilliger. 
	\newblock Tridiagonal pairs of Krawtchouk type.
	\newblock Linear Algebra Appl. 427 (2007) 218--233.
	
	\bibitem{IT2009MMJ}
	T.~Ito, P.~Terwilliger. 
	\newblock Distance-regular graphs of  $q$-Racah type and the $q$-tetrahedron algebra.
	\newblock Michigan Math. J. 58 (2009), no. 1, 241--254.
	
	\bibitem{IT2010JAA}
	T.~Ito, P.~Terwilliger. 
	\newblock How to sharpen a tridiagonal pair.
	\newblock J. Algebra Appl. 9 (2010), no. 4, 543--552.

	
	\bibitem{Lam2001}
	T.~Y.~Lam.
	\newblock A First Course in Noncommutative Rings, 2nd edition.
	\newblock Graduate Texts in Mathematics, Vol.~131. Springer, 2001.
	
	\bibitem{Lee2013}
	J.-H.~Lee.
	\newblock $Q$-polynomial distance-regular graphs and a double affine Hecke algebra of rank one.
	\newblock Linear Algebra Appl. 439 (2013) 3184--3240.
	
	\bibitem{Lee2017}
	J.-H.~Lee.
	\newblock Nonsymmetric Askey-Wilson polynomials and $Q$-polynomial distance-regular graphs.
	\newblock J. Combin. Theory Ser. A, 147 (2017) 75--118.
		
	\bibitem{Morales2025}
	J.~V.~Morales.
	\newblock A rank two Leonard pair in Terwilliger algebras of Doob graphs.
	\newblock J. Combin. Theory Ser. A, 210 (2025) Paper No. 105958, 21pp.
	
	\bibitem{NT2007LAA}
	K.~Nomura, P.~Terwilliger. 
	\newblock The split decomposition of a tridiagonal pair. 
	\newblock Linear Algebra Appl. 424 (2007) 339--345.
	
	\bibitem{NT2008LAA2}
	K.~Nomura, P.~Terwilliger.
	\newblock {Sharp tridiagonal pairs}.
	\newblock Linear Algebra Appl. 429 (2008) 79--99.
	
	\bibitem{NT2008LAA}
	K.~Nomura, P.~Terwilliger.
	\newblock {The structure of a tridiagonal pair}.
	\newblock Linear Algebra Appl. 429 (2008) 1647--1662.
	
	\bibitem{Roman2010}
	S.~Roman.
	\newblock Advanced Linear Algebra, 3rd edition.
	\newblock Graduate Texts in Mathematics, Vol.~135. Springer, New York, 2010.

	\bibitem{Suzuki2005JoAC}
	H.~Suzuki.
	\newblock The Terwilliger algebra associated with a set of vertices in a distance-regular graph.
	\newblock J. Algebraic Combin. 22 (2005), no. 1, 5--38.

	\bibitem{TKCP2022}
	Y.-Y.~Tan, J.~Koolen, M.~Cao, J.~Park.
	\newblock Thin $Q$‑polynomial distance‑regular graphs have bounded $c_2$.
	\newblock Graphs Combin. 38 (2022), no. 6, Paper No. 175, 18 pp.
		

	\bibitem{TW2020}
	H.~Tanaka, T.~Wang.
	\newblock The Terwilliger algebra of the twisted Grassmann graph: the thin case.
	\newblock Electron. J. Combin. 27 (2020) no. 4, Paper No. 4.15, 22 pp.

	\bibitem{Ter1992JAC}
	P.~Terwilliger.
	\newblock {The subconstituent algebra of an association scheme I}.
	\newblock J. Algebraic Combin. 1 (1992) 363--388.
	
	\bibitem{Ter1993JAC-1}
	P.~Terwilliger. 
	\newblock The subconstituent algebra of an association scheme II. 
	\newblock J. Algebraic Combin. 2 (1993) 73--103.
	
	\bibitem{Ter1993JAC-2}
	P.~Terwilliger. 
	\newblock The subconstituent algebra of an association scheme III. 
	\newblock J. Algebraic Combin. 2 (1993) 177--210.
	
	\bibitem{Ter1993SuzukiNote}
	P.~Terwilliger. 
	\newblock Lecture note on Terwilliger algebra (edited by H. Suzuki) 1993.
	
	\bibitem{Ter2001LAA}
	P.~Terwilliger. 
	\newblock Two linear transformations each tridiagonal with respect to an eigenbasis of the other.
	\newblock Linear Algebra Appl. 330 (2001) 149--203.

	\bibitem{Ter2024}
	P.~Terwilliger. 
	\newblock Distance-regular graphs, the subconstituent algebra, and the $Q$-polynomial property. 
	\newblock London Math. Soc. Lecture Note Ser. 487, Cambridge University Press, London, 2024, 430--491.
	
	\bibitem{TA2019}
	P.~Terwilliger, A.~\v Zitnik.
	\newblock The quantum adjacency algebra and subconstituent algebra of a graph.
	\newblock J. Combin. Theory Ser. A 166 (2019) 297--314.

	\bibitem{TY1994}
	M.~Tomiyama, N.~Yamazaki.
	\newblock The subconstituent algebra of a strongly regular graph.
	\newblock Kyushu J. Math. 48, No. 2 (1994) 323--334.
	
 \end{thebibliography}
\end{document}